\documentclass[11pt]{article}
\usepackage[T1]{fontenc}
\usepackage[utf8]{inputenc}
\usepackage{lmodern}
\usepackage[margin=1in]{geometry}
\usepackage{amsmath,amssymb,amsthm,mathtools, times, color}
\usepackage{microtype}
\numberwithin{equation}{section}
\usepackage[colorlinks,linkcolor=blue,anchorcolor=blue,citecolor=blue]{hyperref}
\newcommand{\beq}{\begin{equation}}
\newcommand{\eeq}{\end{equation}}
\newcommand{\beqn}{\begin{eqnarray}}
\newcommand{\eeqn}{\end{eqnarray}}

\newtheorem{theorem}{Theorem}[section]
\newtheorem{proposition}[theorem]{Proposition}
\newtheorem{lemma}[theorem]{Lemma}
\newtheorem{corollary}[theorem]{Corollary}
\theoremstyle{definition}
\newtheorem{definition}[theorem]{Definition}
\newtheorem{remark}[theorem]{Remark}
\newtheorem{conjecture}[theorem]{Conjecture}
\newtheorem{claim}[theorem]{Claim}

\newcommand{\lemref}[1]{Lemma~\ref{#1}}

\DeclareMathOperator{\Ric}{Ric}
\DeclareMathOperator{\Rm}{Rm}
\DeclareMathOperator{\tr}{tr}

\newcommand{\R}{\mathbb R}

\newcommand{\cH}{\mathcal H}
\newcommand{\cC}{\mathcal C}

\newcommand{\vphi}{\varphi}
\newcommand{\om}{\omega}
\newcommand{\Om}{\Omega}
\newcommand{\p}{\partial}

\title{Long-time existence of the Calabi flow
on K\"ahler surfaces \\with bounded Sobolev constants}

\author{Haozhao Li\footnote{Supported by NSFC grant  No. 12471058 and a research fund from  University of Science and Technology of China.}, \quad Bing Wang\footnote{Supported by NSFC grant No. 12431003 and a research fund from University of Science and Technology of China.} \quad and \quad Kai Zheng
\footnote{Supported by NSFC grant No. 12571091.}}

\date{}

\begin{document}
\maketitle

\begin{abstract}
We prove that the Calabi flow extends past any finite time on  compact K\"ahler surfaces if its Sobolev constants remain uniformly bounded. Moreover,  we show the long-time existence of Calabi flow  in certain K\"ahler classes on  compact K\"ahler surfaces when the initial Calabi energy is strictly below an explicit threshold.
\end{abstract}

\begingroup
\small
\tableofcontents
\endgroup

\section{Introduction}
\label{sec:statement}
In \cite{Calabi}
Calabi  introduced the Calabi flow to study canonical metrics in a fixed K\"ahler class, particularly extremal K\"ahler metrics and constant scalar curvature K\"ahler (cscK) metrics. In this paper, we study finite-time extension of the flow on compact K\"ahler surfaces under a uniform bound for the Sobolev constants.

Let $(M,J,\omega)$ be a compact K\"ahler surface, where $J$ is the complex structure and $\omega$ is a K\"ahler form. Write $\Omega=[\omega]$ and let $\mathcal H_\omega$ denote the space of smooth real-valued potentials $\varphi$ such that $\omega_\varphi=\omega+i\partial\bar\partial\varphi>0$. We define the normalised volume form $d\mu_\varphi$ and the volume $V$ by
\begin{align}\label{normalised volume form}
d\mu_\vphi=V^{-1}\omega_\vphi^2,\qquad V=\Om^2=\int_M\omega^2.
\end{align}
Let $R_\varphi$ denote the K\"ahler scalar curvature. In terms of the first Chern class $c_1(M)$, we have
\begin{align}\label{average scalar formhula}
\underline R=\int_M R_\vphi\,d\mu_\vphi
             =\frac{4\pi}{V}\,c_1(M)\cdot\Omega.
\end{align}
Thus $V$ and $\underline R$ depend only on the fixed K\"ahler class and $c_1(M)$. In terms of K\"ahler potentials, we write the Calabi flow as
\beq
\partial_t\varphi=R_\varphi-\underline R,\qquad
\varphi(\cdot,0)=\varphi_0.
\label{eq:exact-Calabi-gauge}
\eeq
The normalised Calabi energy is defined by
\begin{align*}
\cC(\vphi)=\int_M(R_\vphi-\underline R)^2\,d\mu_\vphi.
\end{align*}
Along the flow, this energy is nonincreasing. With respect to the $L^2$ metric on the space of K\"ahler potentials, the Calabi flow is the negative gradient flow of the Mabuchi K-energy \cite{MabuchiEnergy,MabuchiGeometry}. Fine \cite{Fine} gave a projective approximation of the smooth Calabi flow by balancing flows. Despite these geometric and variational interpretations, the fourth-order nature of the equation makes finite-time degeneration difficult to control.

On compact Riemann surfaces, Chru\'sciel \cite{Chrusciel}, Chen \cite{ChenSurfaces}, and Struwe \cite{Struwe} established long-time existence and convergence by different methods. In higher complex dimensions, Chen and He \cite{ChenHeContinuation} developed the short time existence and stability theory near cscK metrics. Huang-Zheng \cite{MR2953047} extended the stability theory near Calabi's extremal metrics. He \cite{MR3010280} and later He and Zeng \cite{HeZeng} subsequently established short time existence and instantaneous smoothing for rough initial data, including continuous K\"ahler metrics. The general long time existence problem remains open. Chen \cite{MR3858468} formulated the following conjecture.
\begin{conjecture}[X.~X.~Chen]\label{Chen conjecture}
For every initial K\"ahler metric on a compact K\"ahler manifold, the Calabi flow exists for all time.
\end{conjecture}
Several authors have obtained curvature criteria for extending the flow. Chen and He \cite[Theorem~1.5]{ChenHeContinuation} proved extension past a finite time under a uniform Ricci curvature bound. By the a priori estimates of Chen and Cheng \cite{ChenCheng}, it suffices to assume a uniform scalar curvature bound. Li, Zhang, and Zheng \cite[Theorem~1.2]{LiZhangZheng} weakened this hypothesis to a uniform $L^p$ bound for the scalar curvature, with $p>n$, where $n$ is the complex dimension. More recently, Li and Zhang \cite{LiZhangII} obtained extension criteria involving suitable space-time integrals of the scalar curvature. Under additional geometric assumptions, long-time existence and convergence are also known for initial metrics with small Calabi energy. Tosatti and Weinkove \cite{TosattiWeinkove} proved such results when the first Chern class vanishes and in certain classes when it is negative. Li and Zheng \cite{LiZheng} obtained further small energy results using geometric bounds and eigenvalue estimates. Assuming global existence, Li, Wang, and Zheng \cite{LiWangZheng} used regularity scales to establish convergence results on extremal K\"ahler surfaces, with analogous results in higher dimensions under an additional scalar curvature bound.

From the variational viewpoint, Streets \cite{Streets} constructed global minimising movements for the K-energy and subsequently proved that they agree with smooth Calabi flow solutions for as long as the latter exist \cite{StreetsConsistency}. Berman \cite{Berman} studied conditional convergence in the canonical and anti-canonical classes. Building on the metric geometry of finite-energy potentials, Berman, Darvas, and Lu \cite{BermanDarvasLu} developed the weak Calabi flow on the finite energy completion and studied its long-time behaviour. These weak existence results do not by themselves establish the smooth extension asserted in Conjecture~\ref{Chen conjecture}.

For singularity analysis under a uniform Sobolev bound, Chen and He \cite{ChenHe} established a particularly useful description in complex dimension two. They proved that a maximal-curvature blow-up yields a complete nonflat scalar-flat K\"ahler asymptotically locally Euclidean (ALE) surface \cite[Theorem~1.1 and Section~4]{ChenHe}. Thus, to obtain a curvature bound, it suffices to exclude such bubbles. In the toric setting, Chen and He \cite{ChenHe,ChenHeToric} excluded bubbles under suitable energy and geometric assumptions, and Huang \cite{Huang} obtained potential estimates under a uniform Sobolev bound. These results suggest studying whether the Sobolev bound alone rules out finite-time singularities on a general compact K\"ahler surface.

In the present paper, we answer this question affirmatively by combining a potential estimate, a localisation argument, and a topological obstruction. For $\varphi\in\mathcal H_\omega$, let $g_\varphi$ be the real Riemannian metric associated with $\omega_\varphi$. Denote by $C_S(\varphi)$ the least constant in the $W^{1,2}\to L^4$ Sobolev inequality \eqref{eq:Sobolev}, with respect to $g_\varphi$ and the probability measure $d\mu_\varphi$. Our main result is the following.

\begin{theorem}
\label{thm:main}
Let $(M, g)$ be a compact K\"ahler surface, and let $\varphi(t)$ be a solution of the Calabi flow on the finite time interval $[0,T)$ with $T<\infty$. Suppose
$$
 \sup_{0\le t<T}C_S(\varphi(t))<\infty.
$$
Then the Calabi flow can be smoothly extended past time $T$.

\end{theorem}

We next estimate the Sobolev constant in terms of the K\"ahler class and the initial Calabi energy. Together with Theorem~\ref{thm:main}, this yields the following long time existence result.

\begin{corollary}
\label{cor:Tian-main}Let $(M, g)$ be a compact K\"ahler surface, and let $\varphi(t)$ be the maximal smooth Calabi flow in the fixed class
$\Omega=[\omega]$ with the initial potential
$\varphi(\cdot,0)=\varphi_0\in\cH_\omega$. Suppose
\beq
 c_1(M)\cdot\Omega>0,\qquad  \delta_\Omega=c_1(M)^2
       -\frac23\frac{(c_1(M)\cdot\Omega)^2}{V}>0,\qquad
 \cC(\varphi_0)<\frac{16\pi^2}{V}\delta_\Omega.
 \label{eq:Tian-main-hypotheses}
\eeq
Then the Calabi flow exists for all $t\ge0$.
\end{corollary}

\begin{remark}
The relevant estimates for the Sobolev constant appear in Chen--LeBrun--Weber \cite[Section~5]{CLW} and in Chen--He \cite[Lemma~2.3]{ChenHeToric}. The cohomological expression $\delta_\Omega$ is the one used in \cite[Appendix]{ChenHe}. In \cite[Theorem 1.2]{ChenHe}, Chen and He also proved long time existence and subsequential convergence to a cscK metric on toric Fano surfaces, assuming a condition on the initial energy. This result was extended to toric extremal metrics in \cite[Theorem 1.1]{ChenHeToric} under an additional condition on the positivity of the extremal Hamiltonian potential. In Corollary~\ref{cor:Tian-main}, we impose no toric invariance assumption. Moreover, the first positivity condition in \eqref{eq:Tian-main-hypotheses} is weaker than the Fano condition $c_1(M)>0$.
\end{remark}

We outline the proof of Theorem~\ref{thm:main}. Suppose that the curvature becomes unbounded. By the blow-up analysis of Chen--He \cite{ChenHe}, we obtain a scalar flat K\"ahler ALE bubble. The curvature decay results of Anderson \cite{MR2148795} and Tian--Viaclovsky \cite{TV}, together with the fact that the complex structure is parallel, allow us to apply the compactification theorem of Hein, R\u{a}sdeaconu, and \c{S}uvaina
\cite{HRS}. It follows that every bounded plurisubharmonic function
on the bubble is constant. Using the bounded potentials, we then show that the pullbacks of the reference K\"ahler form have vanishing mass on compact subsets of the bubble. With respect to the pulled-back complex structures $J_i=F_i^*J$ and the pulled-back rescaled metrics $h_i$, the convergence embeddings $$F_i:(U_i,J_i,h_i)\to(M,J,g)$$ are holomorphic between K\"ahler manifolds and hence harmonic. By applying the local harmonic-map estimate (Lemma~\ref{lem:harmonic}), see Lin \cite{Lin}, we show that the images of compact domains have diameter tending to zero in the fixed reference metric. Thus a compact core of the bubble embeds in a coordinate ball in
$M$, and hence in $S^4$ (cf. Lemma~\ref{lem:localisation}).
Such an embedding forces the second integral homology of the core to vanish and imposes a direct-sum constraint on the first homology of its boundary. By the ALE results of Wright \cite{Wright}, Bando, Kasue, and Nakajima \cite{BKN}, and \c{S}uvaina \cite{Suvaina}, we conclude that a scalar flat K\"ahler ALE surface with such a core must be flat. This contradicts the normalisation of the curvature at the base point of the bubble. Therefore, we obtain a uniform curvature bound and conclude by the extension theorem of Chen--He \cite{ChenHeContinuation}.

For Corollary~\ref{cor:Tian-main}, let $\mathcal Y_t$ be the Yamabe constant of $g_{\varphi(t)}$ and set
$\sigma_t^2=2V\cC(\varphi(t))$. We use the Yamabe minimiser
\cite{Schoen}, the conformal curvature identity in \cite{CLW}, and
Calabi energy monotonicity to obtain a uniform positive lower bound for $\mathcal Y_t-\sigma_t$. It then suffices to absorb the scalar curvature term in the Yamabe inequality, as in \cite[Lemma~2.3]{ChenHeToric}, to obtain the required uniform Sobolev bound. By Theorem~\ref{thm:main}, the maximal existence time is therefore infinite. The argument, including the sign of the Yamabe constant and all normalisation factors, is given in Section~\ref{sec:Tian-cone}.

The paper is organised as follows. In Section~\ref{sec:preliminaries}, we fix notation and recall the Calabi energy and the length estimates. The potential estimate, the local harmonic-map estimate, and the localisation and exclusion of ALE bubbles are established in Section~\ref{Estimates}. Finally, we prove Theorem~\ref{thm:main} and Corollary~\ref{cor:Tian-main} in Section~\ref{Proofs of the main results}.\\

\textbf{On the usage of AI.} The authors used  ChatGPT as an  assisting tool in the development of the paper. The authors take full responsibility for the contents of the paper.\\

\section{Preliminaries}
\label{sec:preliminaries}
\subsection{Notation}
Throughout the paper, we work on K\"ahler surfaces, except when explicitly discussing results in higher dimensions. Our conventions are
$$
 d^c=\frac{i}{2}(\bar\partial-\partial),\quad
 dd^c=i\partial\bar\partial.
$$
With the normalisation in \eqref{normalised volume form}, we have $\mu_\varphi(M)=1$. For $\vphi\in\mathcal H_\om$, let $dV_{g_\vphi}$ denote the Riemannian volume measure. Then
\begin{align}\label{volume form and measure}
 dV_{g_\vphi}=\frac12\omega_\vphi^2=\frac V2\,d\mu_\vphi,\qquad
 \operatorname{Vol}(M,g_\vphi)=\frac V2.
\end{align}
We use the following conventions for the Ricci form and the K\"ahler scalar curvature:
\begin{align}\label{scalar formhula}
 \Ric(\omega_\vphi)=-i\partial\bar\partial\log\det(g_{\vphi,j\bar k}),
 \qquad R_\vphi=\tr_{\omega_\vphi}\Ric(\omega_\vphi).
\end{align}
The real scalar curvature is twice
$R_\vphi$ and
$[\Ric(\omega_\vphi)]=2\pi c_1(M)$.
We use $\Rm(g_\vphi)$ for the real Riemann curvature tensor with its full tensor norm, and $\Ric(g_\vphi)$ for the real Ricci tensor. The complex and real Laplacians are related by
$$
 \Delta_\vphi f=g_\vphi^{j\bar k}\partial_j\partial_{\bar k}f,
 \qquad \Delta^{\R}_{g_\vphi}=2\Delta_\vphi.
$$
We use $\nabla$ for the real Riemannian gradient, so that $|\nabla f|_{\vphi}^2=2|\p f|_{\vphi}^2$ for every smooth function $f$. For brevity, let
$g(t)=g_{\varphi(t)}$ and $d\mu_t=d\mu_{\varphi(t)}$.

\begin{definition}[Sobolev constant] \label{defn:Sobolev}
For $\vphi\in\cH_\omega$, let $C_S(\vphi)$ be the least constant for which
\beq
 \left(\int_M|f|^4\,d\mu_\vphi\right)^{1/2}
 \le C_S(\vphi)\left(\int_M|\nabla f|_{\vphi}^2\,d\mu_\vphi
                         +\int_M|f|^2\,d\mu_\vphi\right)
 \label{eq:Sobolev}
\eeq
holds for every $f\in W^{1,2}(M,g_\vphi)$. Taking $f=1$ shows that $C_S(\vphi)\ge1$.
\end{definition}

\begin{definition} \label{defn:Sobolev-volume}
We say that a Calabi flow has a uniform Sobolev constant bound $C_S$ on the time interval $[0,T)$ if $C_S(\varphi(t))\le C_S$ for every $t\in [0,T)$. In terms of the Riemannian volume measure on $M$, this is equivalent to
\beq
 \begin{split}
 \left(\int_M|f|^4\,dV_{g(t)}\right)^{1/2}
 \le S_0\left(\int_M|\nabla f|_{g(t)}^2\,dV_{g(t)}
                         +\int_M|f|^2\,dV_{g(t)}\right),\quad
 S_0=C_S(V/2)^{-1/2}.
 \end{split}
 \label{eq:Sobolev-volume}
\eeq
\end{definition}

\subsection{The Calabi energy and the $L^2$ distance}
In the following calculation, we take all covariant derivatives with respect to $g_\varphi$. For a smooth real valued function $f$, define
$$
 \mathcal D_\varphi f
 =\bar\partial\bigl(\nabla^{1,0}_{g_\varphi}f\bigr),
 \qquad
 \nabla^{1,0}_{g_\varphi}f
 =g_\varphi^{j\bar k}f_{\bar k}\frac{\partial}{\partial z^j}.
$$
In holomorphic coordinates, the Calabi flow equation gives
$$
 \partial_t g_{\varphi,j\bar k}=R_{\varphi,j\bar k},\qquad
 \partial_t g_\varphi^{j\bar k}
 =-g_\varphi^{j\bar b}R_{\varphi,a\bar b}g_\varphi^{a\bar k},\qquad
 \partial_t\log\det g_\varphi=\Delta_\varphi R_\varphi.
$$
Differentiating $\Ric(\omega_\varphi)=-i\partial\bar\partial\log\det g_\varphi$ and taking the trace, we obtain
\beq
 \begin{split}
 \partial_t R_\varphi
 =-\Delta_\varphi^2R_\varphi
   -\bigl\langle\Ric(\omega_\varphi),i\partial\bar\partial R_\varphi
     \bigr\rangle_{\varphi},\quad
 \partial_t d\mu_t=(\Delta_\varphi R_\varphi)\,d\mu_t.
 \end{split}
 \label{eq:scalar-volume-variation-detailed}
\eeq
By integration by parts, it follows that
\beq
 \frac{d}{dt}\cC(\varphi(t))
 =-2\int_M|\mathcal D_\varphi R_\varphi|_{\varphi}^2\,d\mu_t\le0.
 \label{eq:calabi-monotonicity}
\eeq
For the computation, see \cite[Section~3]{MR2953047}. Consequently,
\begin{align}\label{Calabi energy bound}
\cC(\varphi(t))\le\cC(\varphi_0)
\end{align} throughout the smooth
existence interval.

\begin{lemma}\label{lem:both-endpoint-L2}
Along the Calabi flow $\vphi(t)$, there exists a universal constant $C_D>0$ such that
\beq
 \|\varphi(t)-\varphi_0\|_{L^2(\mu_{\varphi_0})}
   +\|\varphi(t)-\varphi_0\|_{L^2(\mu_t)}
\le C_D\cdot t\cdot \sqrt{\cC(\varphi_0)}.
 \label{eq:both-endpoint-L2}
\eeq
\end{lemma}
\begin{proof}
Let $d_2$ denote the distance induced by the $L^2$ metric on $\cH_\omega$. Along the Calabi flow $\vphi(t)$, we estimate this distance using \eqref{Calabi energy bound}:
\beq
 d_2(\varphi_0,\varphi(t))
 \le\int_0^t\|\partial_s\varphi(s)\|_{L^2(\mu_s)}\,ds
 =\int_0^t\sqrt{\cC(\varphi(s))}\,ds
 \le t\sqrt{\cC(\varphi_0)}.
 \label{eq:finite-time-d2}
\eeq
Note that there exists a universal constant $C$, independent of $\vphi_0$ and $\vphi_1$, such that
\[C^{-1} I_2(\vphi_0,\vphi_1) \leq d^2_2(\vphi_0,\vphi_1)\leq CI_2(\vphi_0,\vphi_1),\quad \forall \vphi_0,\vphi_1\in \mathcal H_\om\]
where the $I_2$ functional is defined as
\[
I_2(\vphi_0,\vphi_1)=\int_M(\varphi_1-\varphi_0)^2 d\mu_{\varphi_0}
   +\int_M(\varphi_1-\varphi_0)^2 d\mu_{\varphi_1},
\]
see \cite[Theorem~3]{Darvas}. Combining this comparison with \eqref{eq:finite-time-d2}, we obtain \eqref{eq:both-endpoint-L2}.
\end{proof}

\section{Estimates}\label{Estimates}

\subsection{A potential estimate}

\begin{lemma}
\label{lem:probability-Moser}
Let $\vphi\in\cH_\omega$ and suppose $C_S(\vphi)\le C_S$. If $f\in C^\infty(M,\R)$ satisfies $\Delta_\vphi f\ge-a$ for some constant $a\geq 1$, then we have
$$
 \|f_+\|_{L^\infty(M)}
 \le 4(a+1)C_S\bigl(1+\|f\|_{L^2(\mu_\vphi)}\bigr),
 \qquad f_+=\max\{f,0\}.
$$
\end{lemma}

\begin{proof}
For $\varepsilon>0$, we set
$$
 \eta_\varepsilon(s)=\frac12\left(s+\sqrt{s^2+\varepsilon^2}\right),
 \qquad u_\varepsilon=1+\eta_\varepsilon(f).
$$
Then $u_\varepsilon\ge1$ and
$$
 \eta_\varepsilon'(s)
 =\frac12\left(1+\frac{s}{\sqrt{s^2+\varepsilon^2}}\right)\in[0,1],
 \qquad
 \eta_\varepsilon''(s)
 =\frac{\varepsilon^2}{2(s^2+\varepsilon^2)^{3/2}}\ge0.
$$
By the assumed inequality for $f$, we have
$$
 \begin{aligned}
 \Delta_{\vphi}  u_\varepsilon
 =\eta_\varepsilon'(f)\Delta_{\vphi} f
       +\eta_\varepsilon''(f)|\p f|_{\vphi}^2\ge-a.
 \end{aligned}
$$
For fixed $\varepsilon$, set $u=u_\varepsilon$. Multiplying this inequality by $u^{p-1}$ for $p\ge2$ and integrating by parts, we obtain
$$
 (p-1)\int_Mu^{p-2}|\p u|_{\vphi}^2\,d\mu_\vphi
 \le a\int_Mu^{p-1}\,d\mu_\vphi.
$$
Since $u\ge1$, it follows that
$$
 \begin{aligned}
\int_M|\p u^{p/2}|_{\vphi}^2\,d\mu_\vphi
 &=\frac{p^2}{4}\int_Mu^{p-2}|\p u|_{\vphi}^2\,d\mu_\vphi\\
 &\le\frac{ap^2}{4(p-1)}\int_Mu^{p-1}\,d\mu_\vphi
 \le\frac{ap^2}{4(p-1)}\int_Mu^p\,d\mu_\vphi.
 \end{aligned}
$$
Applying the Sobolev inequality in Definition~\ref{defn:Sobolev} to $u^{p/2}$, we obtain
$$
 \|u\|_{L^{2p}(\mu_\vphi)}^p
 \le C_S\left(1+\frac{ap^2}{2(p-1)}\right)
           \|u\|_{L^p(\mu_\vphi)}^p
 \le(a+1)p\cdot C_S\|u\|_{L^p(\mu_\vphi)}^p.
$$
In the last inequality, it suffices to use $p\ge2$ and the elementary bound
$$1+\frac{ap^2}{2(p-1)}\le1+\frac{ap^2}{2p/2} \le p+ap = (a+1)p.$$
For the standard iteration argument, let $p_j=2^{j+1}$. Iterating from $p_0=2$ to $p_{N+1}$, we obtain
$$
 \|u\|_{L^{p_{N+1}}(\mu_\vphi)}
 \le\prod_{j=0}^{N}[(a+1)p_jC_S]^{1/p_j}\|u\|_{L^2(\mu_\vphi)}.
$$
Since
$$
 \sum_{j=0}^{\infty}\frac1{p_j}=1,
 \qquad
 \prod_{j=0}^{\infty}p_j^{1/p_j}
 =2^{\sum_{j=0}^{\infty}(j+1)/2^{j+1}}=4,
$$
we may let $N\to\infty$ to obtain
$$
 \|u_\varepsilon\|_{L^\infty(M)}
 \le4(a+1)C_S\|u_\varepsilon\|_{L^2(\mu_\vphi)}.
$$
From $$0\le\eta_\varepsilon(s)-\max\{s,0\}\le\varepsilon/2,$$ we see that
$$u_\varepsilon\to1+f_+\text{ uniformly as }\varepsilon\to0.$$
Since $\mu_\varphi(M)=1$, it follows that
$$\|1+f_+\|_{L^2(\mu_\vphi)}
 \le1+\|f\|_{L^2(\mu_\vphi)}.
$$ The lemma is proved.
\end{proof}

\begin{proposition}
\label{prop:potential}
Let $\varphi(t)$ be a Calabi flow on the K\"ahler surface $M$, with a uniform Sobolev constant bound $C_S$ on
$[0,T)$. Then
\beq
 \|\varphi(t)-\varphi_0\|_{L^\infty(M)}\leq
 C(C_S, C_D)\bigl(1+t\sqrt{\cC(\varphi_0)}\bigr),
 \qquad 0\le t<T,
 \label{eq:finite-time-potential-increment}
\eeq
where $C_D$ is the constant in
\lemref{lem:both-endpoint-L2}, and
$C=12C_S\max\{1,C_D\}$.
In particular, if $T<\infty$, the potential is uniformly bounded on $M\times[0,T)$.
\end{proposition}

\begin{proof}
We set $w(t)=\varphi(t)-\varphi_0$. Then
$
 \omega_{\varphi(t)}=\omega_{\varphi_0}+i\p\bar\p w(t),
$
which implies that
$$
 \begin{aligned}
 \Delta_{\varphi_0}w(t)
   =\operatorname{tr}_{\omega_{\varphi_0}}\omega_{\varphi(t)}-2\ge-2,\quad
 \Delta_{\varphi(t)}(-w(t))
   =\operatorname{tr}_{\omega_{\varphi(t)}}\omega_{\varphi_0}-2\ge-2.
 \end{aligned}
$$
By the uniform Sobolev bound and Lemma~\ref{lem:probability-Moser} with $a=2$, we have
$$
 \begin{aligned}
 \|w(t)_+\|_{L^\infty}
  &\le12C_S\bigl(1+\|w(t)\|_{L^2(\mu_{\varphi_0})}\bigr),\\
 \|w(t)_-\|_{L^\infty}
  &\le12C_S\bigl(1+\|w(t)\|_{L^2(\mu_t)}\bigr),
 \end{aligned}
$$
where $w_+=\max\{w,0\}$ and $w_-=\max\{-w,0\}$. Taking the maximum of these two estimates and using Lemma~\ref{lem:both-endpoint-L2} yields
 \begin{align*}
 \|w(t)\|_{L^\infty}
 &=\max\{\|w(t)_+\|_{L^\infty},\|w(t)_-\|_{L^\infty}\}\\
 &\le12C_S\left(1+
       \max\left\{\|w(t)\|_{L^2(\mu_{\varphi_0})},
                    \|w(t)\|_{L^2(\mu_t)}\right\}\right)\\
 &\le12C_S\left(1+C_Dt\sqrt{\cC(\varphi_0)}\right)\\
 &\le12C_S\max\{1,C_D\}
               \left(1+t\sqrt{\cC(\varphi_0)}\right).
 \end{align*}
This proves \eqref{eq:finite-time-potential-increment}. The final assertion follows immediately when $T<\infty$.
\end{proof}

\begin{remark}
Huang \cite[Theorem 1.1]{Huang} proved the finite time potential bounds for the Calabi flow under a uniform Sobolev bound in the toric setting.
\end{remark}

\subsection{A local harmonic map estimate}

We record a consequence of the epsilon regularity estimate for harmonic maps. See Lin \cite[proof of Lemma~1.5 and Section~5]{Lin}. We state it for smoothly converging domain metrics, as required by the localisation argument below.

\begin{lemma}
\label{lem:harmonic}
Let $U_i\Subset X$ exhaust a smooth real four-manifold and let
$h_i\to h$ smoothly locally on $X$. Let $(N,g_N)$ be a fixed compact
Riemannian manifold and let $$f_i:(U_i,h_i)\to(N,g_N)$$ be smooth harmonic maps. If
\[
 \int_K |df_i|_{h_i,g_N}^2\,dV_{h_i}\longrightarrow0
 \quad\text{for every compact }K\Subset X,
\]
then on every compact subset $$\sup_K|df_i|_{h_i,g_N}\to0.$$
\end{lemma}

\begin{proof}
Let $K\Subset X$ be compact, and let $K'\Subset X$ be a compact neighbourhood with $K\subset\operatorname{int}K'$.
For all sufficiently large $i$, the set $K'$ lies in $U_i$. By the smooth convergence of $h_i$, we have uniform ellipticity, uniform bounds for the metric derivatives, and uniform control of the local geometry near $K'$. Hence there exists a fixed sufficiently small $r\in(0,1)$ such that, for every $p\in K$, the balls $B^{h_i}_{4r}(p)$ have compact closure in
$\operatorname{int}K'$ and lie in uniformly controlled coordinate
neighbourhoods.

\textbf{Step 1: Epsilon regularity.}
We use the small energy gradient estimate recalled in
\cite[p.~792, proof of Lemma~1.5]{Lin}, together with the extension to
Riemannian domain metrics in \cite[Section~5, pp.~825--826]{Lin}.
The following is a uniform local reformulation of that estimate. For a fixed $\Lambda\ge1$ and the fixed compact target $(N,g_N)$, there are constants $\varepsilon_0>0$ and $C_0<\infty$ such that any smooth metric $\widetilde h$ on the Euclidean ball
$B_2(0)\subset\mathbb R^4$ satisfying
\[
 \Lambda^{-1}\delta\le\widetilde h\le\Lambda\delta,
 \qquad
 \max_{a,b}\|\widetilde h_{ab}\|_{C^2(B_2(0))}\le\Lambda
\]
and any smooth harmonic map
$\widetilde f:(B_2(0),\widetilde h)\to(N,g_N)$ satisfy
\begin{align}\label{small energy gradient estimate Lin}
 \int_{B_2(0)}|d\widetilde f|_{\widetilde h,g_N}^2
                  \,dV_{\widetilde h}<\varepsilon_0
 \quad\Longrightarrow\quad
 \sup_{B_1(0)}|d\widetilde f|_{\widetilde h,g_N}^2\le C_0.
\end{align}
Here $\varepsilon_0$ and $C_0$ depend only on $\Lambda$ and $(N,g_N)$, and the domain dimension is four.
In our setting, smooth harmonic maps are stationary under compactly supported domain variations, so the stationarity hypothesis of the cited estimate is satisfied.
To apply this statement, we pull back $f_i$ to the tangent space. For $p\in K$, let $T_pX\simeq\mathbb R^4$ be an $h_i$-orthonormal identification. With this identification, set
\[
 \Phi_{i,p,r}(y)=\exp_p^{h_i}(2ry),
 \qquad
 \widetilde h_i=(2r)^{-2}\Phi_{i,p,r}^*h_i,
 \qquad
 \widetilde f_i=f_i\circ\Phi_{i,p,r}.
\]
Each $\Phi_{i,p,r}$ identifies $B_2(0)$ with $B^{h_i}_{4r}(p)$. By the smooth local convergence of $h_i$, the preceding coordinate bounds hold with a common $\Lambda$ for all sufficiently large $i$ and all $p\in K$. Since harmonicity is preserved under constant rescaling of the source metric, we may apply
\eqref{small energy gradient estimate Lin} to the harmonic map $$\widetilde f_i:(B_2(0),\widetilde h_i)\to(N,g_N).$$ Let $e_i=|df_i|_{h_i,g_N}^2$ denote the energy density of $f_i$.
In real dimension four,
\[
 \int_{B_2(0)}|d\widetilde f_i|_{\widetilde h_i,g_N}^2
                     \,dV_{\widetilde h_i}
 =(2r)^{-2}\int_{B^{h_i}_{4r}(p)}e_i\,dV_{h_i},
 \qquad
 |d\widetilde f_i|_{\widetilde h_i,g_N}^2
 =(2r)^2 e_i\circ\Phi_{i,p,r}.
\]
With $\varepsilon_*=\varepsilon_0/4$, we obtain from the preceding estimate
\begin{align}\label{small energy gradient estimate}
 r^{-2}\int_{B^{h_i}_{4r}(p)}e_i\,dV_{h_i}\le\varepsilon_*
 \quad\Longrightarrow\quad
 \sup_{B^{h_i}_{2r}(p)}e_i\le Cr^{-2}.
\end{align}
Indeed, $C=C_0/4$. Since $r$ is fixed and
$$\int_{K'}e_i\,dV_{h_i}\to0,$$ the smallness condition in \eqref{small energy gradient estimate} holds
simultaneously for every $p\in K$ when $i$ is sufficiently large.

\textbf{Step 2: Bochner formula.}
Let \(\{E_a\}\) be an orthonormal frame with respect to \(h\), and set $df_a=df(E_a)$ and $e=|df|_{h,g_N}^2$. By the Bochner formula of Eells and Sampson \cite[Section~8]{EellsSampson}, we have
\begin{align*}
 \frac12\Delta_h^{\R}e
 =|\nabla df|^2
   +\sum_{a,b}\Ric_h(E_a,E_b)\langle df_a,df_b\rangle_{g_N}
-\sum_{a,b}
       \langle\Rm^N(df_a,df_b)df_b,df_a\rangle_{g_N}.
\end{align*}
The second and third terms can be bounded in absolute value by $Ce$ and $Ce^2$, respectively, with constants depending on the Ricci curvature bounds for $h_i$ and the curvature bounds for $g_N$. Applying this formula to $f_i$, we obtain
\[
 \Delta^{\mathrm R}_{h_i} e_i
 \ge -C_0 e_i-C_1 e_i^2.
\]
These constants can be chosen uniformly on the fixed compact neighbourhood because $h_i$ converges smoothly and the target is fixed and compact. Using the preceding gradient bound \eqref{small energy gradient estimate} and $r<1$, we obtain
\begin{align}\label{Bochner PDE}
 \Delta^R_{h_i}e_i\ge -a_0\cdot r^{-2}\cdot e_i
 \quad\text{on }B^{h_i}_{2r}(p),
\end{align}
where $a_0\ge0$ is independent of $i$, $p$, and the sufficiently small radius $r$.

\textbf{Step 3: Local estimates.}
Using the uniform local bounds for $h_i$, we apply the local $L^2$ boundedness estimate for uniformly elliptic equations in divergence form to \eqref{Bochner PDE}. All constants below depend only on these bounds and $a_0$.
We set
\[
 \widehat h_i=r^{-2}h_i,
 \qquad w=e_i,
 \qquad B_s=B^{\widehat h_i}_s(p)=B^{h_i}_{sr}(p).
\]
Since
\[
 \Delta^R_{\widehat h_i}=r^2\Delta^R_{h_i},
 \qquad dV_{\widehat h_i}=r^{-4}dV_{h_i},
\]
we have
\begin{align}\label{Bochner PDE w}
 w\ge0,\qquad
 \Delta^R_{\widehat h_i}w\ge-a_0\cdot w
 \quad\text{on }B_2.
\end{align}
Let $1\le s<t\le3/2$. For each $x\in B_s$, set $R=\frac{t-s}{2}$. We apply
\cite[Theorem~8.17]{GilbargTrudinger} with $p=2$, $n=4$ to the equation \eqref{Bochner PDE w}:
\[
 w(x)\leq\sup_{B_R(x)}w
 \le CR^{-2}
       \left(\int_{B_{2R}(x)}w^2\,dV_{\widehat h_i}\right)^{1/2}.
\]
Since \(B_{2R}(x)\subset B_t(p)\), taking the supremum over \(x\in B_s(p)\) gives
\[
 \sup_{B_s}w
 \le C(t-s)^{-2}
       \left(\int_{B_t}w^2\,dV_{\widehat h_i}\right)^{1/2}.
\]
For convenience, set \[ A(s)=\sup_{B_s}w,
 \qquad E=\int_{B_2}w\,dV_{\widehat h_i}.
\]
Since $w\geq 0$,
\begin{align}\label{L2 local}
 \int_{B_t}w^2\,dV_{\widehat h_i}
 \le A(t)\int_{B_t}w\,dV_{\widehat h_i}
 \le A(t)E.
\end{align}
Using \eqref{L2 local} in the preceding supremum estimate and applying Young's inequality with $0<\theta<1$, we obtain
\[
 A(s)
 \le C(t-s)^{-2}A(t)^{1/2}E^{1/2}
 \le \theta A(t)+C_\theta(t-s)^{-4}E.
\]
For the iteration, let \[ s_j=\frac32-2^{-j-1},\qquad j=0,1,2,\ldots.
\]
Then $s_0=1$, $s_j\to3/2$, and
$s_{j+1}-s_j=2^{-j-2}$. Iteration gives
\[
 A(1)
 \le\theta^{m+1}A(s_{m+1})
    +C_\theta E\sum_{j=0}^{m}\theta^j2^{4j+8}
 \le\theta^{m+1}A(3/2)
    +C\cdot E\cdot \sum_{j=0}^{m}(16\theta)^j.
\]
For the fixed $i$, $p$, and $r$, smoothness of $w$ on a neighbourhood of $\overline{B_{3/2}}$ implies $A(3/2)<\infty$. Hence the first term tends to zero as $m\to\infty$. For $\theta=1/32$, the series in the second term converges, so
\[
 \sup_{B_1}w\le C\int_{B_2}w\,dV_{\widehat h_i}.
\]
Returning to $h_i$ and recalling that $w=e_i$, we conclude that
\begin{equation}
 \sup_{B^{h_i}_r(p)}e_i
 \le Cr^{-4}\int_{B^{h_i}_{2r}(p)}e_i\,dV_{h_i}
 \le Cr^{-4}\int_{B^{h_i}_{4r}(p)}e_i\,dV_{h_i}.
 \label{eq:harmonic-epsilon}
\end{equation}
The factor $r^{-4}$ comes from the volume rescaling in real dimension four. All constants depend only on the controlled local geometry and the fixed target. In particular, they are independent of the map and the centre $p\in K$ for all sufficiently large $i$. The radius $r>0$ was fixed before taking $i\to\infty$, and the fourfold balls lie in $K'$. Taking the supremum over their centres
$p\in K$ therefore gives
\[
 \sup_K e_i
 \le C_{K,K'}\int_{K'}e_i\,dV_{h_i}\longrightarrow0.
\]
In particular, for all sufficiently large $i$ we have
\[
 \sup_K|df_i|_{h_i,g_N}
 \le C_{K,K'}
 \left(\int_{K'}|df_i|_{h_i,g_N}^2\,dV_{h_i}\right)^{1/2}
 \longrightarrow0.
\]
\end{proof}

\subsection{Localisation of an ALE bubble using bounded potentials}

We first specify the compactification and pluripotential results used below, with the conventions
\[
 d^c_Jv=-\frac12\,dv\circ J,
 \qquad dd^c_Jv=i\partial_J\bar\partial_Jv,
 \qquad \Omega_h(\xi,\zeta)=h(J\xi,\zeta),\qquad \Omega_h(\xi,J\xi)=h(\xi,\xi).
\]

The following definition is from \cite[Definition~2.1]{HRS}.
\begin{definition}\label{ALE defn}
A complete K\"ahler $n$-manifold is asymptotically locally Euclidean (ALE) if
\begin{itemize}
\item outside a compact set, its finitely
many ends admit diffeomorphisms to
\[
 (\mathbb C^n\setminus\overline{B_R(0)})/\Gamma,
 \qquad \Gamma\subset U(n)\text{ finite and acting freely},
\]
\item in the lifted coordinates, for some $\tau>0$ and $0<\alpha<1$,
\[
 h-h_{\mathrm{Euc}},\qquad J-J_{\mathrm{Euc}}
 \in C^{2,\alpha}_{-\tau}.
\]
\end{itemize}
For a tensor $T$, this weighted condition means, componentwise,
\[
 \sum_{j=0}^2 |x|^j|\partial^jT(x)|
 +|x|^{2+\alpha}
 [\partial^2T]_{C^{0,\alpha}(B_{|x|/10}(x))}
 =O(|x|^{-\tau}).
\]
Here derivatives and norms are Euclidean. The H\"older seminorm compares the coordinate components on the indicated ball.
\end{definition}

By the compactification theorem of Hein, R\u{a}sdeaconu, and \c{S}uvaina \cite[Theorem~2.2]{HRS}, for $n\ge2$ we may compactify such a manifold $X$ to a compact K\"ahler orbifold $\overline X$ by adding a connected suborbifold divisor $D$. Moreover, $X$ has one end, and
\[
 X\cong\overline X\setminus D.
\]
We will also need a removable singularity theorem for plurisubharmonic functions.
\begin{definition}
A function on a coordinate domain is plurisubharmonic if it is upper semicontinuous and its restriction to every complex line is subharmonic, allowing the value $-\infty$.

A set is pluripolar if locally it is contained in the $-\infty$ locus of a plurisubharmonic function that is not identically
$-\infty$. For further details, see \cite[Chapter~I, Definitions~5.1 and~5.22]{DemaillyCADG}.
\end{definition}

We use the following removable singularity theorem \cite[Chapter~I, Theorem~5.24]{DemaillyCADG}. Let $A$ be a closed pluripolar subset of a complex manifold $Y$, and let $v$ be a plurisubharmonic function on $Y\setminus A$. Suppose that $v$ is locally bounded above near $A$. Then $v$ extends uniquely to a plurisubharmonic function on $Y$.

\begin{lemma}[Liouville property]
\label{lem:Liouville}
Let $(X,h,J)$ be a connected ALE K\"ahler manifold of complex dimension $n\ge2$. Then every globally bounded plurisubharmonic function on $X$ is constant.
\end{lemma}

\begin{proof}
By the compactification theorem, we identify $X$ biholomorphically with
$\overline X\setminus D$.
Let $f$ be a bounded plurisubharmonic function with $|f|\le B_f$. At a point of $D$, we choose a local covering chart
\[
 \pi:\widetilde W\longrightarrow W=\widetilde W/H,
 \qquad \widetilde D=\pi^{-1}(D\cap W),
\]
where $\widetilde W$ is a complex manifold and $H$ is finite. The lift $\widetilde f=f\circ\pi$ is plurisubharmonic on
$\widetilde W\setminus\widetilde D$ and satisfies the boundedness condition
$|\widetilde f|\le B_f$ there.
After the chart is shrunk if necessary, the hypersurface $\widetilde D$ has a nonzero holomorphic defining function $a$. Since
$\log|a|$ is plurisubharmonic and equals $-\infty$ on
$\widetilde D$, this hypersurface is pluripolar.
It is closed in the chart, and $\widetilde f$ is locally bounded above near it. Thus we have verified all three hypotheses of the removable singularity theorem.
More explicitly, the extension is given by
\[
 \widetilde f^{\mathrm{ext}}(z)=
 \begin{cases}
  \widetilde f(z),&z\notin\widetilde D,\\
  \displaystyle\limsup_{\substack{w\to z, w\notin\widetilde D}}
                    \widetilde f(w),&z\in\widetilde D.
 \end{cases}
\]
The same bounds hold: $-B_f\le\widetilde f^{\mathrm{ext}}\le B_f$. For $\gamma\in H$, both $\widetilde f^{\mathrm{ext}}$ and
$\widetilde f^{\mathrm{ext}}\circ\gamma$ extend $\widetilde f$,
so they agree by uniqueness. Hence the extension descends to
$W$.
On overlaps, the lifts to a common refinement are extensions of the same function, so they again agree by uniqueness. Therefore, we obtain a globally defined bounded plurisubharmonic function
$\overline f$ on the orbifold $\overline X$.
By compactness and upper semicontinuity, $\overline f$ attains its maximum at some point. Its lift to a connected uniformising chart is
subharmonic and attains an interior maximum, so it is constant on that chart by the strong maximum principle. Therefore the set where
$\overline f$ attains its maximum is open. It is also closed by
upper semicontinuity. The space $\overline X$ is connected because it is the closure of the connected dense subset $X$. Thus the maximum set is all of $\overline X$, and $f$ is constant on $X$.
\end{proof}

\begin{lemma}
\label{lem:localisation}
Let $(M,J,g,\omega)$ be a fixed closed K\"ahler surface. Let $\omega_i=\omega+dd^c\psi_i>0$, let $g_i$ be the associated K\"ahler metric, and suppose that
\[
 \sup_M\psi_i=0,\qquad \|\psi_i\|_\infty\le B.
\]
Let $U_i\Subset X$ exhaust a connected smooth real four-manifold $X$. Suppose there are embeddings $F_i$ and constants $Q_i$:
\[F_i:U_i\to M,
 \qquad Q_i>0,\qquad Q_i\longrightarrow\infty\] such that
\[
 h_i=Q_i\cdot F_i^*g_i\longrightarrow h,
 \qquad J_i=F_i^*J\longrightarrow J_\infty \text{ smoothly locally,}
\]
and suppose that $(X,h,J_\infty)$ is an ALE K\"ahler surface satisfying the hypotheses of Lemma~\ref{lem:Liouville}.

Then every connected compact smooth domain $U\Subset X$ embeds smoothly into $\mathbb R^4$. More precisely, after passing to a subsequence,
\[
 \operatorname{diam}_gF_i(U)\longrightarrow0.
\]
\end{lemma}

\begin{proof}
Since $F_i$ are embeddings,
$dF_i$ is invertible at every point of $U_i$, and
\[
 J_i=(dF_i)^{-1}\circ J\circ dF_i,
 \qquad dF_i\circ J_i=J\circ dF_i.
\]
In particular, $F_i:(U_i,J_i)\to(M,J)$ is holomorphic. Let $\Omega_\infty$ be the K\"ahler form of $(h,J_\infty)$. For the K\"ahler form of $(h_i,J_i)$, we have
\[
 \Omega_i=Q_i\cdot F_i^*\omega_i \longrightarrow \Omega_\infty\text{ smoothly on compact subsets}.
\]
Let $\alpha_i=F_i^*\omega$. Then
 \begin{align}
  \label{eq:pullback-potential positivity}
 \alpha_i(\xi,J_i\xi)
 =\omega(dF_i\xi,JdF_i\xi)
 =g(dF_i\xi,dF_i\xi)>0,\qquad \xi\ne0.
 \end{align}
Eells and Sampson \cite[Section~2(C)]{EellsSampson} showed that a holomorphic map between K\"ahler manifolds is harmonic for their associated real Riemannian metrics. Equivalently, its tension field
$\operatorname{tr}(\nabla dF)$ vanishes.
In our application the source is $(U_i,J_i,h_i)$ and the target is the fixed $(M,J,g)$. Thus $F_i$ are harmonic with the fixed target metric $g$, not merely with the varying metric $g_i$. At each point, we choose an $h_i$-orthonormal real unitary frame $(E_1,J_iE_1,E_2,J_iE_2)$ diagonalising the positive
$(1,1)$-form $\alpha_i$. If its two eigenvalues relative to
$\Omega_i$ are $\lambda_1,\lambda_2>0$, then
\[
 \alpha_i\wedge\Omega_i
       =(\lambda_1+\lambda_2)\frac{\Omega_i^2}{2}
       =(\lambda_1+\lambda_2)\,dV_{h_i}.
\]
In this frame, the energy density $e_i=|dF_i|_{h_i,g}^2$ is given by
 \begin{align*}
 e_i
 =\sum_{a=1}^2\bigl(|dF_i(E_a)|_g^2
                       +|dF_i(J_iE_a)|_g^2\bigr)
 =2\sum_{a=1}^2g(dF_iE_a,dF_iE_a)
  =2\sum_{a=1}^2\alpha_i(E_a,J_iE_a)
  =2(\lambda_1+\lambda_2).
 \end{align*}
Hence
\[
 |dF_i|_{h_i,g}^2\,dV_{h_i}=2\alpha_i\wedge\Omega_i.
\]
For the pulled-back potential $u_i=-\psi_i\circ F_i$, the normalisation of $\psi_i$ gives
\begin{align}\label{eq:pullback-potential bounded}
0\le u_i\le B.
\end{align}
After passing to a subsequence, we obtain the weak-* convergence
 \begin{align}
  \label{eq:pullback-potential weak star convergence}
 \int_Xu_i\cdot\zeta\cdot dV_h\longrightarrow\int_Xu\cdot\zeta\cdot dV_h
 \quad\text{for every compactly supported }\zeta\in L^1(X,dV_h).
 \end{align}
Moreover, $0\le u\le B$ almost everywhere. Since $F_i$ is holomorphic, we compute using \eqref{eq:pullback-potential positivity} that
 \begin{align}
  \label{eq:pullback-potential}
 dd^c_{J_i}u_i
 =-F_i^*(dd^c_J\psi_i)
  =-F_i^*(\omega_i-\omega)
 =-Q_i^{-1}\Omega_i+\alpha_i
 \geq -Q_i^{-1}\Omega_i.
 \end{align}
To complete the proof, it suffices first to show that $u$ has a bounded plurisubharmonic representative, which must then be constant by Lemma~\ref{lem:Liouville}. It then follows that the local energies of the harmonic maps $F_i$ tend to zero. Finally, by Lemma~\ref{lem:harmonic}, we obtain the required local control of their images.

\textbf{Step 1: Positivity of $dd^c_{J_\infty}u$.}
Let $W\Subset X$ be a coordinate chart and let $\xi$ be a smooth real vector field on $W$. We write
\[
 L_i^\xi v=(dd^c_{J_i}v)(\xi,J_i\xi),
 \qquad
 L_\infty^\xi v=(dd^c_{J_\infty}v)(\xi,J_\infty\xi).
\]
\begin{claim}For every smooth real vector field $\xi$ and every nonnegative $\chi\in C_c^\infty(W)$, we have
\begin{align}\label{Linfty positive}
 \langle L_\infty^\xi u,\chi\rangle \geq 0.
\end{align}
\end{claim}
\begin{proof}
We write
$J(\partial_b)=J^a{}_b\partial_a$. Then
\[
 d^c_Jv=-\frac12J^a{}_b\,\partial_av\,dx^b,
 \qquad
 dd^c_Jv=-\frac12\partial_b(J^a{}_c\,\partial_av)
                         \,dx^b\wedge dx^c.
\]
Thus we can express $L_i^\xi$ and its formal adjoint as
\[
 L_i^\xi v=a_i^{ab}\partial_a\partial_bv+b_i^a\partial_av,\quad
  (L_i^\xi)^*\chi
 =m^{-1}\partial_a\partial_b(m\cdot a_i^{ab}\chi)
  -m^{-1}\partial_a(m\cdot b_i^a\chi),
\]
where $dV_h=m(x)\,dx$ with a fixed smooth positive density $m$. Note that the coefficients depend smoothly on $J_i$, its first derivatives, and $\xi$. By the assumed smooth convergence of $J_i$, we obtain smooth convergence of the coefficients of $L_i^\xi$ and its adjoint on compact subsets. In particular,
\begin{align}\label{eq:pullback-potential adjoint convergence}
 L_i^\xi v \to L_\infty^\xi v,\quad
(L_i^\xi)^*\chi\to(L_\infty^\xi)^*\chi \text{ smoothly locally}.
\end{align}
The difference of the pairings is
 \begin{align*}
 \int_Xu_i(L_i^\xi)^*\chi\,dV_h
       -\int_Xu(L_\infty^\xi)^*\chi\,dV_h
=\int_Xu_i\big[(L_i^\xi)^*-(L_\infty^\xi)^*\big]
                     \chi\,dV_h
       +\int_X(u_i-u)(L_\infty^\xi)^*\chi\,dV_h.
 \end{align*}
By the uniform bound \eqref{eq:pullback-potential bounded} and the convergence \eqref{eq:pullback-potential adjoint convergence}, the first term tends to zero. For the second term, we use the weak-* convergence \eqref{eq:pullback-potential weak star convergence}. Hence
\[
 \langle L_\infty^\xi u,\chi\rangle
 =\lim_{i\to\infty}\int_X\chi L_i^\xi u_i\,dV_h.
\]
Using \eqref{eq:pullback-potential} and $Q_i\rightarrow\infty$, we have
\[
 \int_X\chi\cdot L_i^\xi u_i\,dV_h
 \ge-Q_i^{-1}\int_X\chi\cdot \Omega_i(\xi,J_i\xi)\,dV_h
 =-Q_i^{-1}\int_X\chi \cdot h_i(\xi,\xi)\,dV_h\longrightarrow0.
\]
This proves \eqref{Linfty positive}.
\end{proof}
Furthermore, in a
$J_\infty$-holomorphic chart, we write
$
 \xi=\sum_{j=1}^2
       \left(\zeta^j\partial_{z_j}
+\overline{\zeta^j}\partial_{\bar z_j}\right).
$
By \eqref{Linfty positive}, we have
\[ 0\leq L_\infty^\xi u=
(dd^c_{J_\infty}u)(\xi,J_\infty\xi) =2\sum_{j,k=1}^2\zeta^j\overline{\zeta^k}
                         \partial_j\partial_{\bar k}u
\] in the sense of distributions.

\textbf{Step 2: Constancy of the plurisubharmonic representative of $u$.}
We apply the theorem of Demailly \cite[Chapter~I, Theorem~5.8]{DemaillyCADG} to \(u\), using the positivity of its directional complex Hessians established in \eqref{Linfty positive}. Thus there exists a unique locally integrable plurisubharmonic function $\tilde u$ representing $u$ such that $\tilde u=u$ almost everywhere and
\begin{align}  \label{eq:pullback-potential weak star convergence psh}
\int_W\tilde u\cdot \chi\cdot dV_h=\int_Wu\cdot  \chi\cdot dV_h,\quad \forall \chi\in C_c^\infty(W).
\end{align}
As in the proof of that theorem, we obtain $\tilde u$ as the decreasing limit of the regularisations $u*\rho_\varepsilon$, which satisfy
\[
 0\le u*\rho_\varepsilon\le B,
 \qquad dd^c(u*\rho_\varepsilon)
               =(dd^cu)*\rho_\varepsilon\ge0
\]
where $\rho_\varepsilon$ is a standard nonnegative radial smoothing kernel of integral one in a holomorphic chart. Hence $0\le\tilde u\le B$ everywhere.
By Lemma~\ref{lem:Liouville}, we conclude that
\[
\tilde u\equiv c\quad\text{on }X,\qquad 0\le c\le B.
\]

\textbf{Step 3: Vanishing of the harmonic-map energy.}
\begin{claim}
\begin{equation}
 \int_K\alpha_i\wedge\Omega_i\longrightarrow0
 \qquad\text{for every compact }K\Subset X.
 \label{eq:vanishing-current-mass}
\end{equation}
\end{claim}
\begin{proof}
Let $\chi\in C_c^\infty(X)$. For all large $i$, its support lies in $U_i$. Wedging~\eqref{eq:pullback-potential} with $\chi\Omega_i$ and integrating by parts, we obtain
\begin{equation}
 \int_X\chi\,\alpha_i\wedge\Omega_i
 = \int_X\chi\,dd^c_{J_i}u_i\wedge\Omega_i
        +Q_i^{-1}\int_X\chi\Omega_i^2
 =\int_Xu_i\,dd^c_{J_i}\chi\wedge\Omega_i
        +Q_i^{-1}\int_X\chi\Omega_i^2.
 \label{eq:current-cutoff}
\end{equation}
We write
\[
 dd^c_{J_i}\chi\wedge\Omega_i=b_i\,dV_h,
 \qquad dd^c_{J_\infty}\chi\wedge\Omega_\infty=b_\infty\,dV_h.
\]
Then $b_i\to b_\infty$ smoothly, with common compact support, and
\[
 \begin{aligned}
 \int_Xu_i b_i\,dV_h-c\int_Xb_\infty\,dV_h
 &=\int_Xu_i(b_i-b_\infty)\,dV_h
      +\int_X(u_i-c)b_\infty\,dV_h\longrightarrow0.
 \end{aligned}
\]
By the uniform bound \eqref{eq:pullback-potential bounded}, the first term tends to zero. For the second term, we use the weak-* convergence \eqref{eq:pullback-potential weak star convergence} and the identification \eqref{eq:pullback-potential weak star convergence psh} of $u$ with its constant plurisubharmonic representative. Since $d\Omega_\infty=0$, it follows from Stokes' theorem for compactly supported forms that
\[
\int_Xb_\infty\,dV_h= \int_Xdd^c_{J_\infty}\chi\wedge\Omega_\infty
 =\int_Xd\big(d^c_{J_\infty}\chi\wedge\Omega_\infty\big)=0.
\]
Also, $Q_i^{-1}\int_X\chi\Omega_i^2\to0$ by smooth local convergence and $Q_i\to\infty$. Consequently, \eqref{eq:current-cutoff} becomes
\[
 \int_X\chi\cdot \alpha_i\wedge\Omega_i\longrightarrow0.
\]
For fixed compact $K\Subset X$, let $\chi\in C_c^\infty(X)$ satisfy $\chi\ge0$ and $\chi\ge1$ on $K$. By \eqref{eq:pullback-potential positivity}, we have
\[
 0\le\int_K\alpha_i\wedge\Omega_i
       \le\int_X\chi\cdot \alpha_i\wedge\Omega_i\longrightarrow0.
\]
This proves the claim.
\end{proof}

\textbf{Step 4: Local control of the image.}
By \eqref{eq:vanishing-current-mass} and the identity $|dF_i|_{h_i,g}^2\,dV_{h_i}=2\alpha_i\wedge\Omega_i$, we have the vanishing local energy required in Lemma~\ref{lem:harmonic}. Since $h_i\to h$ smoothly locally, the domains exhaust $X$, and $(M,g)$ is fixed and compact, it follows from that lemma that
\begin{equation}
 \sup_K|dF_i|_{h_i,g}\longrightarrow0,
 \qquad K\Subset X.
 \label{eq:vanishing-differentials}
\end{equation}
For a connected compact smooth domain $U\Subset X$, we use the intrinsic path distance
\[
 d^{\mathrm{intr}}_{h_i,U}(x,y)
 =\inf_{\substack{\gamma\subset U\\\gamma(0)=x,\ \gamma(1)=y}}
          \int_0^1|\dot\gamma(t)|_{h_i}\,dt,
\]
where the curves are piecewise smooth. This is the distance using paths constrained to $U$, rather than the ambient distance on $X$. Its diameter for $h$ is finite. Indeed, finitely many coordinate balls and boundary half-balls cover $U$ and give uniform local path length bounds. By connectedness, there are chains of these sets joining any two points.
Smooth convergence on a compact neighbourhood of $U$ implies, for large $i$,
\[
 \tfrac12h\le h_i\le2h,
 \qquad
 \operatorname{diam}^{\mathrm{intr}}_{h_i}(U)
 \le\sqrt2\,\operatorname{diam}^{\mathrm{intr}}_h(U)<\infty.
\]
For any path $\gamma\subset U$ from $x$ to $y$, the operator norm of $dF_i$ is at most its Hilbert--Schmidt norm, so
 \begin{align*}
 d_g(F_i(x),F_i(y))
 \le\operatorname{Length}_g(F_i\circ\gamma)
 \le\sup_U|dF_i|_{h_i,g}\cdot
                  \operatorname{Length}_{h_i}(\gamma).
 \end{align*}
We take the infimum over paths and then the supremum over their endpoints, and use \eqref{eq:vanishing-differentials} to obtain
 \begin{align}\label{eq:bound image of U}
 \operatorname{diam}_gF_i(U)
 \le\sup_U|dF_i|_{h_i,g}\cdot
              \operatorname{diam}^{\mathrm{intr}}_{h_i}(U)
 \longrightarrow0
 \end{align}
Let $x_0\in U$ be fixed. By compactness of $M$, we may pass to a subsequence such that $$F_i(x_0)\to q\in M.$$ There are a coordinate neighbourhood $W_q$ of $q$ and $\varepsilon>0$ such that
$B^g_\varepsilon(q)\subset W_q$. For large $i$, this convergence and \eqref{eq:bound image of U} imply that
\[
 d_g(F_i(x_0),q)<\varepsilon/2,
 \qquad\operatorname{diam}_gF_i(U)<\varepsilon/2.
\]
By the triangle inequality, $F_i(U)\subset W_q$. Composing the smooth embedding $F_i|_U$ with the coordinate chart
$W_q\to\mathbb R^4$ proves the required embedding.
\end{proof}

\subsection{The topological obstruction}

The preceding localisation argument gives an embedding of a compact core of the bubble into a coordinate ball, and hence into $S^4$. We now explain the obstruction to such an embedding. The integral homology of the boundary is essential: the vanishing of $b_2$ alone will give Ricci-flatness, but not flatness.

\begin{lemma}
\label{lem:topology}
Let $(X,h,J)$ be a connected complete smooth scalar-flat K\"ahler ALE four-manifold of positive order. Let $U$ be a connected compact smooth core such that
\begin{enumerate}
\item The exterior of $U$ has the product structure
\[
 X\setminus\operatorname{int}U\cong Y\times[0,\infty),
 \qquad Y=\partial U=S^3/\Gamma.
\]
\item
$U$ embeds smoothly in $S^4$.
\end{enumerate}
Then $h$ is flat.
\end{lemma}

\begin{proof}
All homology and cohomology groups have integer coefficients unless a coefficient group is explicitly displayed. We use the complex orientation of $X$ and the induced boundary orientation of $Y$. A compact core here is a compact codimension-zero smooth submanifold with the displayed product description of its exterior. In particular, collapsing the exterior collar onto $Y$ gives a deformation retraction of $X$ onto $U$.

\textbf{Step 1: The topology of the boundary $Y$.}
We identify $U$ with its image under the assumed embedding into $S^4$. The group $\Gamma$ is finite and acts freely and orientation-preservingly on $S^3$. Thus $Y$ is a connected closed oriented smooth three-manifold.
Since $S^3$ is simply connected, the covering $S^3\to Y$ is universal, and
\[
 \pi_1(Y)\cong\Gamma,
 \qquad H_1(Y)\cong\Gamma/[\Gamma,\Gamma].
\]
In particular, $H_1(Y)$ is finite.
We use the following form of Alexander duality
\cite[Corollary~3.45]{Hatcher}: for a nonempty proper compact
locally contractible subset $K\subset S^n$,
\[
 \widetilde H_j(S^n\setminus K;\mathbb Z)
 \cong\widetilde H^{n-j-1}(K;\mathbb Z).
\]
Here $K=Y$ is compact and locally contractible because it is a smooth closed manifold.
Taking $n=4$ and $j=0$ and using the orientation of $Y$, we obtain
\[
 \widetilde H_0(S^4\setminus Y;\mathbb Z)
 \cong H^3(Y;\mathbb Z)\cong\mathbb Z.
\]
The complement is an open manifold, so its connected components are path components. By the preceding computation of reduced zeroth homology, there are exactly two components.
The interior of $U$ is connected. It is both open and closed in
$S^4\setminus Y$: openness follows from the codimension-zero embedding,
and closedness follows because $U$ is closed and all of its boundary lies in $Y$. Thus it is one entire complementary component.
Let $W$ be the closure of the other component. The smooth embedded boundary has a two-sided collar, so $W$ is a connected compact smooth oriented four-manifold, and
\[
 S^4=U\cup W,\qquad U\cap W=\partial U=\partial W=Y.
\]
By the universal coefficient theorem in degree one, we have
\[
 H^1(Y;\mathbb Z)\cong
 \operatorname{Hom}(H_1(Y;\mathbb Z),\mathbb Z)=0,
\]
since $H_1(Y)$ is finite and $H_0(Y)=\mathbb Z$ has no
$\operatorname{Ext}$ term.
Applying Poincar\'e duality to the closed oriented three-manifold $Y$, we obtain
\[
 H_2(Y;\mathbb Z)\cong H^1(Y;\mathbb Z)=0.
\]
Together with $$H_0(Y)=H_3(Y)=\mathbb Z,$$ this also explains why $Y$ is a rational homology three-sphere.

\textbf{Step 2: Apply Mayer--Vietoris and identify its maps.}
For an open cover $Z=P\cup Q$, the homology Mayer--Vietoris sequence contains
\[
 \cdots\longrightarrow H_j(P\cap Q)
 \xrightarrow{(i_*,-j_*)}H_j(P)\oplus H_j(Q)
 \longrightarrow H_j(Z)
 \longrightarrow H_{j-1}(P\cap Q)\longrightarrow\cdots,
\]
see \cite[Section~2.2, pp.~149--150]{Hatcher}.
In the present situation, we thicken $U$ and $W$ slightly across their common collar. This gives an open cover of $S^4$ whose two members retract onto $U$ and $W$ and whose intersection retracts onto $Y$. Hence the sequence applies to
$S^4=U\cup_YW$ with these homology groups.
The relevant segment is
\[
\begin{aligned}
 H_2(Y)&\longrightarrow H_2(U)\oplus H_2(W)
 \longrightarrow H_2(S^4)\longrightarrow H_1(Y)\\
 &\xrightarrow{(i_*,-j_*)}H_1(U)\oplus H_1(W)
 \longrightarrow H_1(S^4).
\end{aligned}
\]
By \textbf{Step~1}, $H_2(Y)=H_2(S^4)=H_1(S^4)=0$, so exactness gives
\begin{equation}
 \begin{gathered}
 H_2(U)=H_2(W)=0,\qquad H_1(Y)\cong A\oplus B,\\
 A=H_1(U),\qquad B=H_1(W).
 \end{gathered}
 \label{eq:MV}
\end{equation}
In particular, both $A$ and $B$ are finite.
More precisely, the isomorphism
\begin{align}\label{H1AB}
H_1(Y)\to A\oplus B
\end{align} is the map induced by the two
inclusions, with a minus sign on the second component. Therefore
$$i_*:H_1(Y)\to H_1(U)$$ corresponds to projection onto $A$, and
\[
 \ker i_*\cong B.
\]
This identification of the kernel, not just an abstract direct-sum isomorphism, is needed in the next step.

\textbf{Step 3: Use duality to prove that the two summands agree.}
The homology exact sequence of the pair $(U,Y)$ contains
\[
 H_2(U)\longrightarrow H_2(U,Y)
 \xrightarrow{\partial}H_1(Y)
 \xrightarrow{i_*}H_1(U).
\]
Since $H_2(U)=0$, the boundary map $\partial$ is injective. By exactness its image is $\ker i_*$, so
\[
 H_2(U,Y)\cong B.
\]
We recall precisely the duality used here. If $P$ is a compact oriented $m$-manifold with boundary, cap product with its relative fundamental class gives
\[
 H^q(P;\mathbb Z)\cong H_{m-q}(P,\partial P;\mathbb Z),
 \qquad
 H^q(P,\partial P;\mathbb Z)\cong H_{m-q}(P;\mathbb Z).
\]
These are the two boundary cases of \cite[Theorem~3.43, p.~254]{Hatcher}. The hypotheses hold for $P=U$, with $m=4$ and $\partial U=Y$.
With $q=2$ in the first isomorphism, we obtain
\[
 H_2(U,Y)\cong H^2(U).
\]
The cohomological universal coefficient theorem
\cite[Theorem~3.2, p.~195]{Hatcher} gives, for a space $P$,
\[
 0\longrightarrow
 \operatorname{Ext}(H_{q-1}(P;\mathbb Z),\mathbb Z)
 \longrightarrow H^q(P;\mathbb Z)
 \longrightarrow
 \operatorname{Hom}(H_q(P;\mathbb Z),\mathbb Z)
 \longrightarrow0.
\]
We apply this with $P=U$ and $q=2$. The last term vanishes by
\eqref{eq:MV}. Consequently,
\[
 H^2(U)\cong\operatorname{Ext}(A,\mathbb Z).
\]
For a finite abelian group $A$ there is an isomorphism
$$\operatorname{Ext}(A,\mathbb Z)\cong A.$$ To check this explicitly, we write $$A=\bigoplus_{j=1}^m\mathbb Z/q_j\mathbb Z.$$ The free resolution
\[
 0\longrightarrow\mathbb Z\xrightarrow{\ \times q_j\ }
 \mathbb Z\longrightarrow\mathbb Z/q_j\mathbb Z
 \longrightarrow0
\]
shows that
\[
 \operatorname{Ext}(\mathbb Z/q_j\mathbb Z,\mathbb Z)
 \cong\operatorname{coker}
 (\mathbb Z\xrightarrow{\ \times q_j\ }\mathbb Z)
 \cong\mathbb Z/q_j\mathbb Z.
\]
The assertion follows by additivity over finite direct sums.
Combining the preceding identifications, we obtain
\[
 B\cong H_2(U,Y)\cong H^2(U)
 \cong\operatorname{Ext}(A,\mathbb Z)\cong A.
\]
Substitution into \eqref{H1AB} gives
\begin{equation}
 H_1(Y;\mathbb Z)\cong A\oplus A.
 \label{eq:double-group}
\end{equation}
The group structure in
\eqref{eq:double-group} will provide the final contradiction.

\textbf{Step 4: Ricci-flatness.}
The exterior product retracts onto its boundary by
$$(y,r)\mapsto(y,(1-s)r),\quad 0\le s\le1,$$ while $U$ remains fixed.
It follows that $U\hookrightarrow X$ is a homotopy equivalence, and
\eqref{eq:MV} in \textbf{Step 2} gives
\begin{align}\label{b2=0}
 H_2(X;\mathbb Z)=0,\qquad b_2(X)=0,\qquad H^2(X;\mathbb R)=0.
\end{align}
For the last equality we use the universal coefficient theorem over the field $\mathbb R$.
The Ricci-flatness assertion of \cite[Corollary~C]{Wright} states:
\emph{a scalar-flat K\"ahler ALE four-manifold of any positive order
with $b_2=0$ is Ricci-flat.} All its hypotheses have now been verified.
For clarity, we give the analytic mechanism from the proof of that corollary in \cite[Section~8]{Wright}.
Let $\Omega_h$ be the K\"ahler form and let
$\varrho=\operatorname{Ric}(\Omega_h)$ be the Ricci form. It is a
real closed $(1,1)$-form. Scalar-flatness gives
\[
 d\varrho=0,\qquad\operatorname{tr}_{\Omega_h}\varrho=0.
\]
By the second equality, $\varrho$ is primitive, i.e. $\varrho\wedge \Omega_h=0$. On a K\"ahler surface, real primitive $(1,1)$-forms are anti-self-dual in the complex orientation. Hence
\[
 *_h\varrho=-\varrho,
 \qquad d_h^*\varrho=-*_hd*_h\varrho=*_hd\varrho=0.
\]
Thus $\varrho$ is harmonic.
Positive ALE order $\tau>0$ means that, in lifted coordinates on the end, the metric has the decay
\[
 \partial^j(h_{ab}-\delta_{ab})=O(r^{-\tau-j})
 \quad (j=0,1,2),
\]
with the corresponding higher regularity in the smooth ALE formulation. From the coordinate formula for curvature, we obtain $|\varrho|_h\le C|\operatorname{Rm}(h)|_h=O(r^{-2-\tau})$.
The volume density is comparable to the Euclidean one. Therefore
\[
 \int_{\{r\ge r_0\}}|\varrho|_h^2\,dV_h
 \le C\int_{r_0}^{\infty}r^{-4-2\tau}r^3\,dr
 =\frac{C}{2\tau}r_0^{-2\tau}<\infty.
\]
On the compact part the integral is finite by smoothness. Thus
$\varrho$ is an $L^2$ harmonic two-form on $X$.
Here the required Hodge result is the middle-degree ALE consequence of \cite[Theorem~1A]{HHM}, used explicitly in the proof of
\cite[Corollary~C]{Wright}. If
\[
 \mathcal H^2_{(2)}(X,h)
 =\{\beta\in C^\infty(\Lambda^2T^*X):
       d\beta=0,\ d_h^*\beta=0,\ \|\beta\|_{L^2(h)}<\infty\},
\]
then
\[
 \mathcal H^2_{(2)}(X,h)
 \cong\operatorname{Im}
 \bigl(H_c^2(X;\mathbb R)\longrightarrow H^2(X;\mathbb R)\bigr).
\]
Equivalently, the right-hand side is
$$\operatorname{Im}(H^2(U,Y;\mathbb R)\to H^2(U;\mathbb R)).$$
Hausel, Hunsicker, and Mazzeo formulate Theorem~1A of \cite{HHM} for scattering metrics. In dimension four, its middle-degree formula is exactly this image of relative in absolute cohomology. Its use for positive-order ALE metrics can also be justified by choosing a complete metric $h_0$ that is exactly conical outside a compact set:
\[
 h_0=dr^2+r^2g_Y
     =\frac{dx^2}{x^4}+\frac{g_Y}{x^2},\qquad x=r^{-1},
\]
where $g_Y$ is the round quotient metric on $Y=S^3/\Gamma$.
The last expression is the scattering model at the radial boundary
$x=0$. The metrics $h$ and $h_0$ are uniformly equivalent, so they
give equivalent $L^2$ norms in every degree and the same reduced
$L^2$ de Rham cohomology, where ``reduced'' means that the space of
exact forms is replaced by its $L^2$ closure. On a complete manifold this cohomology is represented by $L^2$ harmonic forms.
Applying Theorem~1A to $h_0$ therefore gives the stated vanishing consequence for $h$. This is the ALE application already made in \cite{Wright}.
Because $H^2(X;\mathbb R)=0$, the displayed image is zero. Hence
$\varrho=0$ and
\[
 \operatorname{Ric}(h)=0.
\]

\textbf{Step 5: Anti-self-dual and decay estimates.}
Suppose, for a contradiction, that $h$ is nonflat.
On a scalar-flat K\"ahler surface, we have anti-self-duality in the complex orientation, or equivalently, vanishing of the self-dual Weyl tensor. This follows from the K\"ahler formula for the three eigenvalues of $W^+$, all proportional to the real scalar curvature, see \eqref{CGB W+} and the argument below. Thus $(X,h)$ is a complete Ricci-flat anti-self-dual ALE four-manifold of positive order.
We use the four-dimensional decay consequence of \cite{BKN} in the form recorded in \cite[Definition~1 and Remark~2]{Wright}:
\emph{a Ricci-flat anti-self-dual ALE four-manifold admits coordinates
at infinity in which it is ALE of order four.} In these coordinates,
\[
 \partial^j(h_{ab}-\delta_{ab})=O(r^{-4-j}),\qquad j\ge0.
\]
The ALE K\"ahler convention of \cite[Definition~1.1]{Suvaina} also requires $J-J_0=O(r^{-4})$ in the same end coordinates, where
$J_0$ is a constant Euclidean complex structure and the group at
infinity preserves $J_0$.
For completeness, we verify this decay directly. Since $\nabla^hJ=0$, its components obey
\[
 \partial_aJ^b{}_c
 =-\Gamma^b_{ad}J^d{}_c+\Gamma^d_{ac}J^b{}_d.
\]
Here the indexed $\Gamma^b_{ad}$ denote Christoffel symbols, not the finite group at infinity. Metric decay gives
$\Gamma^b_{ad}=O(r^{-5})$. Compatibility with $h$ gives bounded
coordinate components of $J$, so $|\partial J|=O(r^{-5})$. Along a ray $r\vartheta$ this derivative is integrable, and
\[
 |J(r\vartheta)-J_0(\vartheta)|
 \le C\int_r^\infty s^{-5}\,ds=\frac C4r^{-4}.
\]
For two directions on the covering sphere, we choose a spherical path between the corresponding points of radius $r$ with length at most $Cr$. Then
\[
 |J(r\vartheta_1)-J(r\vartheta_2)|
 \le Cr\sup_{|x|=r}|\partial J(x)|\le Cr^{-4}.
\]
Letting $r\to\infty$ shows that the limiting tensor is independent of direction. We denote it by $J_0$. Passing to the limit in
$J^2=-\operatorname{Id}$ and $h(J\cdot,J\cdot)=h$ shows that $J_0$
is an orthogonal complex structure for the Euclidean metric. By passing to the limit in the equivariance relation for $J$, we obtain equivariance of $J_0$. Therefore, after a constant orthogonal coordinate change, the group at infinity lies in $U(2)$.
We have therefore proved
\[
 h-h_{\mathrm{Euc}}=O(r^{-4}),\qquad J-J_0=O(r^{-4}),
\]
which verifies \cite[Definition~1.1]{Suvaina}. Differentiating the parallel-tensor equation also yields
$$\partial^j(J-J_0)=O(r^{-4-j}),$$ although the zeroth-order conclusion
already verifies the displayed convention in that definition.

\textbf{Step 6: The classification contradicts the boundary homology.}
By \textbf{Steps~4--5}, $h$ is Ricci-flat and the decay estimates required in
\cite[Definition~1.1]{Suvaina} are satisfied. Recall that we are
assuming $h$ is nonflat and have proved $b_2(X)=0$.
If $X$ were simply connected, the Ricci-flat K\"ahler ALE classification would make it hyperk\"ahler. Every nonflat member of that classification has nonzero second Betti number. See
\cite[Theorem~B and Remark~1.2]{Suvaina} and
\cite[Section~2]{Wright}. Hence $X$ is not simply connected.
The non-simply-connected classification in
\cite[Theorem~B and Sections~2--4]{Suvaina} gives
\[
 \pi_1(X)\cong\mathbb Z/k\mathbb Z,\qquad k\ge2,
\]
and a cyclic group at infinity of type
$\frac1{dk^2}(1,dk\ell-1)$ for some $d\ge1$ and
$\gcd(k,\ell)=1$. In particular,
\[
 H_1(Y;\mathbb Z)\cong\mathbb Z/(dk^2)\mathbb Z \text{ is cyclic. }
\]
On the other hand, $U$ is a deformation retract of $X$. Since first homology is the abelianisation of the fundamental group,
\[
 A=H_1(U;\mathbb Z)
   \cong H_1(X;\mathbb Z)
   \cong\mathbb Z/k\mathbb Z.
\]
The boundary-homology identity \eqref{eq:double-group} therefore implies
\[
 H_1(Y;\mathbb Z)
   \cong\mathbb Z/k\mathbb Z\oplus\mathbb Z/k\mathbb Z.
\]
For $k\ge2$ this group is not cyclic: it has $k^2$ elements, whereas the order of each individual element divides $k$.
This contradicts the cyclicity supplied by the classification. Therefore the assumption that $h$ is nonflat is impossible, and $h$ is flat.
\end{proof}

\subsection{Global curvature bounds}

We first bound the curvature energy in terms of the topology and the Calabi energy. The next task is to exclude a maximal-curvature bubble by verifying the hypotheses of the localisation and topological lemmas, using the uniform Sobolev inequality and the bounded oscillation of the potentials after holomorphic pullback.

\begin{lemma}
\label{lem:L2-curvature}
Let $\varphi(t)$, $t\in[0,T)$, be a Calabi flow on a compact K\"ahler surface. There is a constant $\Lambda$, depending only on the topology of $M$, the fixed K\"ahler class, and
$\mathcal C(\varphi_0)$, such that
\[
 \int_M|\operatorname{Rm}(g(t))|_{g(t)}^2\,dV_{g(t)}
 \le\Lambda,\qquad 0\le t<T.
\]
Here,
\[
 \Lambda=2V\bigl(\mathcal C(\varphi_0)+\underline R^2\bigr)
           +32\pi^2|\chi(M)|+96\pi^2|\tau(M)|.
\]
\end{lemma}

\begin{proof}
\textbf{Step 1: Curvature identities.}
For a fixed time $t$, write $q=g(t)$. In this paragraph, we use real Riemannian curvatures and norms and set
\[
 E_q=\operatorname{Ric}(q)-\frac14R_{\mathrm{real}}(q)q,
 \qquad
 I_t=\int_M R_{\mathrm{real}}(q)^2\,dV_q.
\]
Here $E_q$ is the trace-free Ricci tensor.
With the complex orientation, the Hodge star on real two-forms has eigenvalues $+1$ and $-1$, and
\[
 \Lambda^2T^*M=\Lambda^+\oplus\Lambda^-.
\]
We regard $W_q^+$ and $W_q^-$ as the Weyl curvature endomorphisms of these rank-three bundles, with their Hilbert--Schmidt norms.
Let $|\operatorname{Rm}(q)|_q$ denote the norm of the full covariant curvature tensor, as in Section~2.1. The orthogonal algebraic curvature decomposition, with the tensor norm for $\operatorname{Rm}$ and operator norms for $W_q^\pm$, is
\begin{align}\label{The orthogonal algebraic curvature decomposition}
 |\operatorname{Rm}(q)|_q^2
 =4\left(|W_q^+|_q^2+|W_q^-|_q^2\right)
   +2|E_q|_q^2+\frac16R_{\mathrm{real}}(q)^2.
\end{align}
The coefficient $4$ is the conversion from the norm of a curvature endomorphism on two-forms to the norm of its covariant four-tensor: in an orthonormal frame, summing over ordered index pairs counts each pair of unordered skew pairs four times.
For a closed oriented Riemannian four-manifold, the Chern--Gauss--Bonnet formula and the Hirzebruch signature formula, in these conventions, are
 \begin{equation}\label{CGB}
 \begin{split}
 8\pi^2\chi(M)
 &=\int_M\left(
 |W_q^+|_q^2+|W_q^-|_q^2
 -\frac12|E_q|_q^2+\frac1{24}R_{\mathrm{real}}(q)^2
 \right)dV_q,\\
 12\pi^2\tau(M)
 &=\int_M\left(|W_q^+|_q^2-|W_q^-|_q^2\right)dV_q.
  \end{split}
 \end{equation}
Here $\chi(M)$ and $\tau(M)$ are the Euler characteristic and the signature. See
\cite[Section~6]{ChenHe} and \cite[Section~5]{CLW}.
For a K\"ahler surface in its complex orientation, the eigenvalues of $W_q^+$ are
 \begin{equation*}
 \begin{split}
 \frac{R_{\mathrm{real}}(q)}6,
 \qquad -\frac{R_{\mathrm{real}}(q)}{12},
 \qquad -\frac{R_{\mathrm{real}}(q)}{12}.
  \end{split}
 \end{equation*}
Consequently,
 \begin{equation}\label{CGB W+}
 \begin{split}
 |W_q^+|_q^2
 =\left(\frac1{36}+\frac1{144}+\frac1{144}\right)
 R_{\mathrm{real}}(q)^2
 =\frac1{24}R_{\mathrm{real}}(q)^2,
 \qquad
 \int_M|W_q^+|_q^2\,dV_q=\frac{I_t}{24}.
  \end{split}
 \end{equation}
In particular, scalar-flatness of a K\"ahler surface implies $W_q^+=0$, which is the anti-self-duality condition used later.

\textbf{Step 2: Eliminate the Weyl and trace-free Ricci integrals.}
By the signature formula, we first obtain
\[
 \int_M|W_q^-|_q^2\,dV_q
 =\frac{I_t}{24}-12\pi^2\tau(M).
\]
Adding the two Weyl integrals, we obtain
\[
 \int_M\left(|W_q^+|_q^2+|W_q^-|_q^2\right)dV_q
 =\frac{I_t}{12}-12\pi^2\tau(M).
\]
Substituting into the Gauss--Bonnet formula, we obtain
\[
 8\pi^2\chi(M)
 =\frac{I_t}{12}-12\pi^2\tau(M)
   -\frac12\int_M|E_q|_q^2\,dV_q+\frac{I_t}{24}.
\]
Since $1/12+1/24=1/8$, rearrangement gives
\[
 \int_M|E_q|_q^2\,dV_q
 =\frac{I_t}{4}-16\pi^2\chi(M)-24\pi^2\tau(M).
\]
Integrating the orthogonal algebraic curvature decomposition \eqref{The orthogonal algebraic curvature decomposition} and using the preceding formulas, we find
\[
 \begin{aligned}
\int_M|\operatorname{Rm}(q)|_q^2\,dV_q
 &=4\left(\frac{I_t}{12}-12\pi^2\tau(M)\right)
+2\left(\frac{I_t}{4}-16\pi^2\chi(M)
                     -24\pi^2\tau(M)\right)+\frac{I_t}{6}\\
 &=I_t-32\pi^2\chi(M)-96\pi^2\tau(M).
 \end{aligned}
\]

\textbf{Step 3: Insert the normalised Calabi energy.}
Our conventions are
\[
 R_{\mathrm{real}}(q)=2R_{\varphi(t)},
 \quad dV_q=\frac V2\,d\mu_t,
 \quad \int_Md\mu_t=1,\quad
  \int_M(R_{\varphi(t)}-\underline R)\,d\mu_t=0.
\]
Expanding the square therefore gives
 \begin{align*}
 I_t
 &=2V\int_M R_{\varphi(t)}^2\,d\mu_t
 =2V\int_M\bigl[(R_{\varphi(t)}-\underline R)+\underline R\bigr]^2
                     \,d\mu_t\\
 &=2V\bigl[\mathcal C(\varphi(t))+\underline R^2\bigr]
 \le2V\bigl[\mathcal C(\varphi_0)+\underline R^2\bigr],
 \end{align*}
where the last inequality is \eqref{eq:calabi-monotonicity}. The class is fixed, so $V=\Omega^2$ and
$\underline R=4\pi c_1(M)\cdot\Omega/V$ are fixed as well.
This completes the proof.
\end{proof}

\begin{proposition}
\label{prop:bubble-exclusion}
\label{prop:gauged-bubble}
Let $0<T\le\infty$ and let $\omega_t$ be a Calabi flow on a closed K\"ahler surface $M$, with uniformly bounded Sobolev constants in the sense of Definition~\ref{defn:Sobolev-volume}. Suppose that for every $t<T$ there are a holomorphic automorphism $a_t\in\operatorname{Aut}_0(M,J)$ and a potential $\psi_t$ satisfying
\[
 a_t^*\omega_t=\omega+dd^c\psi_t>0,
 \qquad \operatorname{osc}_M\psi_t\le B<\infty.
\]
Then
\[
 \sup_{M\times[0,T)}|\operatorname{Rm}(g_t)|_{g_t}<\infty.
\]
\end{proposition}

\begin{proof}
\textbf{Step 1: Parabolic rescaling.}
Suppose, to the contrary, that the curvature is unbounded on
$M\times[0,T)$. The function
\[
 t\longmapsto\max_{x\in M}|\operatorname{Rm}(g_t)|_{g_t}(x)
\]
is continuous on every compact subinterval of $[0,T)$. Choosing the first times at which it reaches a sequence of increasing levels tending to infinity gives points $p_i$ and times $t_i\to T$ with
\[
 Q_i=|\operatorname{Rm}(g_{t_i})|_{g_{t_i}}(p_i)
 =\max_{M\times[0,t_i]}|\operatorname{Rm}(g_t)|_{g_t}
 \longrightarrow\infty.
\]
For $T=\infty$, the assertion $t_i\to T$ means $t_i\to\infty$. The times cannot accumulate at a smaller finite time, since the flow is smooth there.
We introduce the rescaled time variable $s$ and set
\[
 g_i(s)=Q_i \cdot g\bigl(t_i+Q_i^{-2}s\bigr),
 \qquad -Q_i^2t_i\le s\le0.
\]
This is the fourth-order parabolic scaling. Indeed, if
$\omega_i(s)=Q_i\omega_{t_i+Q_i^{-2}s}$, then
\[
 R_{\omega_i(s)}=Q_i^{-1}R_{\omega_{t_i+Q_i^{-2}s}},
 \qquad
 \partial_s\omega_i
 =Q_i^{-1}dd^cR_{\omega_{t_i+Q_i^{-2}s}}
 =dd^cR_{\omega_i(s)}.
\]
Thus the rescaled family is again a Calabi flow. Its curvature satisfies
\[
 |\operatorname{Rm}(g_i(s))|_{g_i(s)}\le1,
 \qquad
 |\operatorname{Rm}(g_i(0))|_{g_i(0)}(p_i)=1.
\]
Since $Q_i^2t_i\to\infty$, any fixed interval $[-A,0]$, $A>0$, lies in the rescaled existence interval for large $i$.

\textbf{Step 2: Blow-up limit.}
We use
\cite[Theorem~1.1]{ChenHe}:
a smooth Calabi flow on a compact K\"ahler manifold with uniformly bounded Sobolev constants and unbounded curvature admits a maximal-curvature rescaling subsequence converging smoothly in the pointed Cheeger--Gromov sense to a complete noncompact ancient Calabi flow. In complex dimension two this limit is stationary, scalar-flat, and ALE, with finite total curvature energy.
In summary, we have a connected smooth scalar-flat limit $(X,h,J_\infty,p)$, an exhaustion $U_i\Subset X$, and pointed embeddings $$F_i:U_i\to M$$ such that, after passing to a subsequence,
\[
 h_i:=Q_i\cdot F_i^*g_{t_i}\longrightarrow h,
 \qquad J_i:=F_i^*J\longrightarrow J_\infty
\]
smoothly on every compact subset, with
\[
F_i(p)=p_i,\quad  |\operatorname{Rm}(h)|_h(p)=1.
\]
The last equality follows from smooth convergence and the normalisation in Step~1, and ensures that $h$ is nonflat.

\textbf{Step 3: Curvature energy on the limit.}
For a constant $Q>0$ and $\widehat g=Q\cdot g$ in real dimension four,
\[
 dV_{\widehat g}=Q^2dV_g,
 \qquad |\nabla f|_{\widehat g}^2=Q^{-1}|\nabla f|_g^2,
 \qquad
 |\operatorname{Rm}(\widehat g)|_{\widehat g}^2
 =Q^{-2}|\operatorname{Rm}(g)|_g^2.
\]
For the last equality, the Levi-Civita connection is unchanged by constant scaling, its covariant curvature four-tensor acquires one factor $Q$, and the four inverse metrics used in its squared norm contribute $Q^{-4}$. Thus
\[
 \int_A|\operatorname{Rm}(\widehat g)|_{\widehat g}^2\,dV_{\widehat g}
 =\int_A|\operatorname{Rm}(g)|_g^2\,dV_g
\]
for every measurable set $A$.
Let $K\Subset X$ be a compact domain. For large $i$, we have $K\subset U_i$, so Lemma~\ref{lem:L2-curvature} gives
 \begin{align*}
 \int_K|\operatorname{Rm}(h_i)|_{h_i}^2\,dV_{h_i}
 =\int_{F_i(K)}|\operatorname{Rm}(Q_ig_{t_i})|_{Q_ig_{t_i}}^2
                           \,dV_{Q_ig_{t_i}}
 =\int_{F_i(K)}|\operatorname{Rm}(g_{t_i})|_{g_{t_i}}^2
                           \,dV_{g_{t_i}}
 \le\Lambda,
 \end{align*}
since the energy is invariant under pullback and scaling. Smooth convergence lets us pass to the limit on this fixed compact domain:
\[
 \int_K|\operatorname{Rm}(h)|_h^2\,dV_h\le\Lambda.
\]
An increasing compact exhaustion of $X$ and monotone convergence therefore give
\[
 \int_X|\operatorname{Rm}(h)|_h^2\,dV_h\le\Lambda.
\]

\textbf{Step 4: Derive the homogeneous Sobolev inequality.}
The assumed inequality on the original flow is
\[
 \left(\int_M|f|^4\,dV_{g_t}\right)^{1/2}
 \le S_0\left(
 \int_M|\nabla f|_{g_t}^2\,dV_{g_t}
 +\int_Mf^2\,dV_{g_t}\right),
\]
where $S_0=C_S(V/2)^{-1/2}$ is independent of $t$, as in
\eqref{eq:Sobolev-volume}. Under $\widehat g=Qg$, each term transforms
as follows:
 \begin{align*}
\int_M|f|^4\,dV_{\widehat g}
 =Q^2\int_M|f|^4\,dV_g,\quad
 \int_M|\nabla f|_{\widehat g}^2\,dV_{\widehat g}
 =Q\int_M|\nabla f|_g^2\,dV_g,\quad
 \int_M f^2\,dV_{\widehat g}=Q^2\int_Mf^2\,dV_g.
 \end{align*}
Multiplying the original inequality at $t_i$ by $Q_i$ consequently gives
\[
 \left(\int_M|f|^4\,dV_{Q_ig_{t_i}}\right)^{1/2}
 \le S_0\left(
 \int_M|\nabla f|_{Q_ig_{t_i}}^2\,dV_{Q_ig_{t_i}}
 +Q_i^{-1}\int_Mf^2\,dV_{Q_ig_{t_i}}\right).
\]
To pass to the limit, let $f\in C_c^\infty(X)$ be fixed. For all sufficiently large $i$, its support lies in the interior of $U_i$. Define
\[
 f_i=
 \begin{cases}
 f\circ F_i^{-1},&\text{on }F_i(U_i),\\
 0,&\text{on }M\setminus F_i(U_i).
 \end{cases}
\]
This extension is smooth: $F_i(\operatorname{supp}f)$ is a compact subset of the open image $F_i(U_i)$, so the function already vanishes near the boundary of that image. Applying the last inequality to
$f_i$ and pulling it back gives
\[
 \left(\int_X|f|^4\,dV_{h_i}\right)^{1/2}
 \le S_0\left(
 \int_X|\nabla f|_{h_i}^2\,dV_{h_i}
 +Q_i^{-1}\int_X f^2\,dV_{h_i}\right).
\]
All integrals here are supported in one fixed compact subset of $X$.
The smooth convergence of the metrics implies convergence of the inverse metrics and volume densities there. The first two integrals therefore converge to their $h$-counterparts, whereas the last term tends to zero. We obtain
\[
 \left(\int_X|f|^4\,dV_h\right)^{1/2}
 \le S_0\int_X|\nabla f|_h^2\,dV_h,
 \qquad f\in C_c^\infty(X).
\]
The same inequality holds for compactly supported Lipschitz functions by approximation on their compact support.
The noncompact Sobolev constant used by Tian--Viaclovsky is the best constant in
\[
 \|f\|_{L^4(X,h)}\le C_{\mathrm{hom}}\|\nabla f\|_{L^2(X,h)},
 \qquad f\in C_c^{0,1}(X).
\]
Thus we obtain $C_{\mathrm{hom}}\le S_0^{1/2}$ in the notation of \cite[(1.5)]{TV}. This is a homogeneous inequality, with no
$L^2$ term on the right.

\textbf{Step 5: Apply the ALE decay theorem.}
We use the following specialisation of
\cite[Theorem~5.1]{TV}, namely its anti-self-dual case and
Sobolev-constant alternative: Let $(Z,q)$ be a complete noncompact oriented Riemannian four-manifold with zero scalar curvature and $W_q^+=0$. If
$\int_Z|\operatorname{Rm}(q)|_q^2\,dV_q<\infty$ and its homogeneous
Sobolev constant is finite, then it has finitely many ends, each ALE of order $\tau$ for every $0<\tau<2$.
Here $X$ is complete, smooth, noncompact, and four-dimensional by
\textbf{Step~2}. Its metric is scalar-flat K\"ahler, hence $W_h^+=0$ by the
K\"ahler identity in the preceding Lemma \ref{lem:L2-curvature}. \textbf{Step~3} gives finite total curvature energy, and \textbf{Step~4} gives the required homogeneous Sobolev inequality.
For clarity, an ALE end of order $\tau$ in this Riemannian result has the form $(\mathbb R^4\setminus\overline{B_R})/\Gamma$, where
$\Gamma\subset SO(4)$ is finite and acts freely on $S^3$. On the
covering exterior region, there are coordinates in which the metric coefficients and all their derivatives approach the Euclidean ones. This is the convention in \cite[Section~1, (1.3)--(1.4)]{TVALE}. Fixing any $0<\tau<2$, we obtain
\begin{equation}
 \partial^\ell(h_{ab}-\delta_{ab})=O(r^{-\tau-\ell}),
 \qquad \ell\ge0.
 \label{eq:metric-decay}
\end{equation}
Here $r=|x|$ and $\partial^\ell$ denotes an arbitrary Euclidean partial derivative of order $\ell$ on the covering region.
The theorem gives metric decay. To use the ALE K\"ahler definition fixed in Definition~\ref{ALE defn}, we must still verify decay of $J_\infty$ in the same coordinates.

\textbf{Step 6: Obtain the limiting constant complex structure.}
We lift $h$ and $J_\infty$ to an ALE covering region. Because
$(h,J_\infty)$ is K\"ahler,
\begin{align*}
 \nabla^hJ_\infty=0.
\end{align*}
In coordinates, this identity reads
\begin{align}\label{The parallel-tensor equation}
 \partial_a(J_\infty)^b{}_c
 =-\Gamma^b_{ad}(J_\infty)^d{}_c
   +\Gamma^d_{ac}(J_\infty)^b{}_d,
\end{align}
where $\Gamma^b_{ad}$ are Christoffel symbols. Their coordinate formula is
\[
 \Gamma^b_{ad}
 =\frac12 h^{be}
 \bigl(\partial_a h_{de}+\partial_d h_{ae}-\partial_e h_{ad}\bigr).
\]
The metric bounds~\eqref{eq:metric-decay} and the corresponding inverse metric bounds imply
\[
 \partial^\ell\Gamma^b_{ad}=O(r^{-1-\tau-\ell}),\qquad\ell\ge0.
\]
Also, $J_\infty$ is an $h$-isometry, so its pointwise squared norm as an endomorphism is $4$. Since $h$ is uniformly equivalent to the Euclidean metric at infinity, all coordinate components of $J_\infty$ are bounded. From \eqref{The parallel-tensor equation}, we get
\[
 |\partial J_\infty|=O(r^{-1-\tau}).
\]
Let $\vartheta\in S^3$ be fixed. For $r_2>r_1$ sufficiently large,
 \begin{align*}
 |J_\infty(r_2\vartheta)-J_\infty(r_1\vartheta)|
 &\le C\int_{r_1}^{r_2}s^{-1-\tau}\,ds
 \le\frac C\tau r_1^{-\tau}.
 \end{align*}
Thus the components have a limit along each ray, denoted
$J_0(\vartheta)$, and
\[
 |J_\infty(r\vartheta)-J_0(\vartheta)|
 \le\frac C\tau r^{-\tau}.
\]
For two directions, we join $r\vartheta_1$ and $r\vartheta_2$ by a spherical path of length at most $\pi r$. Then
\[
 |J_\infty(r\vartheta_1)-J_\infty(r\vartheta_2)|
 \le\pi r\sup_{|x|=r}|\partial J_\infty(x)|
 \le Cr^{-\tau}.
\]
Letting $r\to\infty$ proves that $J_0(\vartheta)$ is independent of
$\vartheta$. We denote this constant matrix by $J_0$.
Passing to the limit in $J_\infty^2=-\operatorname{Id}$ and
$h(J_\infty\cdot,J_\infty\cdot)=h$ gives
\[
 J_0^2=-\operatorname{Id},
 \qquad \delta(J_0\cdot,J_0\cdot)=\delta.
\]
Thus $J_0$ is an orthogonal Euclidean complex structure. The lifted tensor is equivariant under the group at infinity:
\[
 J_\infty(\gamma x)=\gamma J_\infty(x)\gamma^{-1},
 \qquad\gamma\in\Gamma.
\]
Taking the limit along rays gives
$J_0=\gamma J_0\gamma^{-1}$. Therefore every $\gamma$ preserves
$J_0$. In a constant orthonormal basis adapted to $J_0$, the group
lies in $U(2)$ and $(\delta,J_0)$ is the standard flat K\"ahler structure.

\textbf{Step 7: Higher derivatives, weighted H\"older decay, and one end.}
Differentiating \eqref{The parallel-tensor equation} proves inductively that
\[
 \partial^\ell(J_\infty-J_0)=O(r^{-\tau-\ell}),\qquad\ell\ge0.
\]
Indeed, the case $\ell=0$ was proved in \textbf{Step~6}. For $\ell\ge1$, each term in the differentiated equation is a product of a derivative of a Christoffel symbol of order $a$ and a derivative of $J_\infty$ of order $b$, where $$a+b=\ell-1.$$ When $b=0$, its size is
$O(r^{-\tau-\ell})$ because $J_\infty$ is bounded. When $b\ge1$,
the induction hypothesis makes it $O(r^{-\ell-2\tau})$, which decays at least as fast.
We have already proved the derivative terms for both
$T=h-\delta$ and $T=J_\infty-J_0$ in Definition~\ref{ALE defn}. To check the seminorm, let
$y,z\in B_{r/10}(x)$. The segment between them lies in that ball,
and all its points have radius comparable to $r$. The mean-value estimate gives
\[
 \frac{|\partial^2T(y)-\partial^2T(z)|}{|y-z|^\alpha}
 \le |y-z|^{1-\alpha}
       \sup_{B_{r/10}(x)}|\partial^3T|
 \le Cr^{1-\alpha}r^{-3-\tau}.
\]
Taking the supremum over $y\ne z$, we obtain
\[
 [\partial^2T]_{C^{0,\alpha}(B_{r/10}(x))}
 \le Cr^{-2-\alpha-\tau}.
\]
Multiplication by $r^{2+\alpha}$ gives precisely the missing
$O(r^{-\tau})$ term.
Thus we have verified that the blow-up limit is a connected ALE K\"ahler manifold in the sense of Definition~\ref{ALE defn}. From \cite[Theorem~2.2]{HRS}, it has one end and is biholomorphic to the complement of a connected complex divisor in a compact K\"ahler orbifold. We may therefore apply Lemma~\ref{lem:Liouville} to $(X,h,J_\infty)$. Let a sufficiently large radius in the single end be fixed. Its coordinate sphere descends to $$Y=S^3/\Gamma,$$ and its complement on the bounded side is a connected compact smooth core $U$. The exterior radial coordinate gives
\[
 X\setminus\operatorname{int}U\cong Y\times[0,\infty).
\]
This is the precise core structure required by Lemma~\ref{lem:topology}.

\textbf{Step 8: Match the holomorphic pullbacks to the localisation lemma.}
At the selected times, we use the automorphisms supplied by the hypothesis and set
\[
 \widehat\omega_i=a_{t_i}^*\omega_{t_i},
 \qquad \widehat g_i=a_{t_i}^*g_{t_i},
 \qquad G_i=a_{t_i}^{-1}\circ F_i.
\]
The equality for $\widehat g_i$ follows from holomorphicity of
$a_{t_i}$: pulling back the K\"ahler form and keeping $J$ fixed gives
exactly the pulled-back Riemannian metric. Each $G_i$ is an embedding on the same domain $U_i$ as $F_i$.
Pullback reverses the order of composition. Consequently,
 \begin{align*}
 Q_iG_i^*\widehat g_i
 =Q_iF_i^*(a_{t_i}^{-1})^*(a_{t_i}^*g_{t_i})
 =Q_iF_i^*g_{t_i}=h_i.
 \end{align*}
Since $(a_{t_i}^{-1})^*J=J$, we also have
\[
 G_i^*J=F_i^*(a_{t_i}^{-1})^*J=F_i^*J=J_i.
\]
Thus the rescaled metric and pulled-back complex structure limits are unchanged.
Adding a constant to $\psi_t$ does not change $dd^c\psi_t$. We therefore replace $\psi_t$ by $\psi_t-\sup_M\psi_t$. The oscillation assumption gives
\[
 -B\le\psi_t\le0,
 \qquad \sup_M\psi_t=0,
 \qquad \|\psi_t\|_{L^\infty}\le B.
\]
At the selected times,
\[
 \widehat\omega_i=\omega+dd^c\psi_{t_i}>0.
\]
We apply Lemma~\ref{lem:localisation} with the fixed reference data
$(M,J,g,\omega)$, the varying metrics $\widehat g_i$, potentials
$\psi_{t_i}$, scales $Q_i$, and embeddings $G_i$.
Here the potentials have the required normalisation and uniform bound, and $Q_i\to\infty$. By the preceding identities and \textbf{Step~7}, we also have the required smooth local convergence and ALE structure. Therefore the core $U$ embeds smoothly into a coordinate ball of $M$, and hence into $\mathbb R^4$. Composing this embedding with the inverse stereographic identification
$\mathbb R^4\cong S^4\setminus\{\text{point}\}$ gives a smooth
embedding $U\hookrightarrow S^4$.
The space $(X,h,J_\infty)$ is connected, complete, smooth, scalar-flat K\"ahler, and ALE of positive order, with the core and exterior collar specified in \textbf{Step~7}. All hypotheses of Lemma~\ref{lem:topology} are therefore satisfied, and that lemma forces $h$ to be flat. This contradicts
\[
 |\operatorname{Rm}(h)|_h(p)=1.
\]
The assumed unbounded curvature is impossible, proving the proposition.
\end{proof}

\section{Proofs of the main results}
\label{Proofs of the main results}

We first combine the potential estimate and the bubble-exclusion proposition to prove finite time extension. It then remains to derive the required Sobolev bound from the hypotheses of Corollary~\ref{cor:Tian-main}, with the real scalar curvature and volume conventions kept explicit throughout.

\subsection{Proof of Theorem~\ref{thm:main}}

\begin{proof}[Proof of Theorem~\ref{thm:main}]
\textbf{Step 1: Verify the hypotheses of the curvature bound.}
By hypothesis there is a finite constant $C_S$ such that
$C_S(\varphi(t))\le C_S$ for all $0\le t<T$. Proposition~\ref{prop:potential}
therefore gives
\[
 \sup_{0\le t<T}\|\varphi(t)\|_{L^\infty(M)}
 \le \|\varphi_0\|_{L^\infty(M)}
       +C(C_S,C_D)\bigl(1+T\sqrt{\mathcal C(\varphi_0)}\bigr)
 <\infty.
\]
In particular,
\[
 \sup_{0\le t<T}\operatorname{osc}_M\varphi(t)
 \le 2\sup_{0\le t<T}\|\varphi(t)\|_{L^\infty(M)}<\infty.
\]
We apply Proposition~\ref{prop:gauged-bubble} to the original flow, with $a_t=\operatorname{Id}$ and $\psi_t=\varphi(t)$. The identity
\[
 a_t^*\omega_{\varphi(t)}
 =\omega+dd^c\psi_t>0
\]
and the oscillation and Sobolev bounds verify all its hypotheses. It follows that
\[
 \sup_{M\times[0,T)}
 |\operatorname{Rm}(g_{\varphi(t)})|_{g_{\varphi(t)}}<\infty.
\]
Thus the real Ricci curvatures are uniformly bounded as well.

\textbf{Step 2: Apply Chen--He's extension theorem.}
We use the finite-time form of Chen--He's extension theorem
\cite[Theorem~1.5]{ChenHeContinuation}.
Our flow has smooth initial metric $\omega_{\varphi_0}$, and \textbf{Step~1} supplies the required Ricci bound. Hence there are $\varepsilon>0$ and a smooth family of K\"ahler forms $\widehat\omega_t$,
$0\le t<T+\varepsilon$, satisfying the Calabi flow and agreeing with
$\omega_{\varphi(t)}$ for $t<T$. The K\"ahler class stays fixed because
\[
 \frac{d}{dt}[\widehat\omega_t]
 =[dd^cR_{\widehat\omega_t}]=0.
\]
Its average K\"ahler scalar curvature is therefore the same constant $\underline R$.
To retain the exact initial value, we define
\[
 \widehat\varphi(x,t)
 =\varphi_0(x)+\int_0^t
       \bigl(R_{\widehat\omega_s}(x)-\underline R\bigr)\,ds,
 \qquad 0\le t<T+\varepsilon.
\]
The integrand is smooth on each compact time interval. Thus
$\widehat\varphi$ is smooth, and spatial differentiation may be
passed under the integral. Since $dd^c\underline R=0$,
 \begin{align*}
 \omega+dd^c\widehat\varphi(t)
 =\omega+dd^c\varphi_0+
      \int_0^tdd^cR_{\widehat\omega_s}\,ds
 =\omega_{\varphi_0}+
      \int_0^t\partial_s\widehat\omega_s\,ds
 =\widehat\omega_t>0.
 \end{align*}
Consequently,
\[
 \partial_t\widehat\varphi
 =R_{\widehat\omega_t}-\underline R
 =R_{\widehat\varphi}-\underline R,
 \qquad \widehat\varphi(\cdot,0)=\varphi_0.
\]
For $t<T$, integrating the original equation
\eqref{eq:exact-Calabi-gauge} gives
\[
 \varphi(x,t)=\varphi_0(x)+
   \int_0^t\bigl(R_{\omega_{\varphi(s)}}(x)-\underline R\bigr)\,ds
 =\widehat\varphi(x,t).
\]
Thus the original potential extends smoothly past $T$, with no change of its initial value or time-dependent renormalisation.
\end{proof}

\subsection{Proof of Corollary~\ref{cor:Tian-main}}
\label{sec:Tian-cone}

\begin{proof}[Proof of Corollary~\ref{cor:Tian-main}]
Let $[0,T)$, $0<T\le\infty$, be the maximal smooth existence interval. As in the statement and Section~\ref{sec:preliminaries}, we write
\[
 g_t=g_{\varphi(t)},\qquad R_t=R_{\varphi(t)},\qquad
 dV_t=dV_{g_t}=\frac12\omega_{\varphi(t)}^2
              =\frac V2\,d\mu_t,
 \qquad \mathcal C_0=\mathcal C(\varphi_0).
\]
We shall prove the Sobolev estimate on this entire interval without assuming that $T$ is finite. It will then follow from Theorem~\ref{thm:main} that the maximal existence time is infinite.

\textbf{Step 1: Convert the energy hypothesis into a fixed gap.}
The hypotheses \eqref{eq:Tian-main-hypotheses} are
\[
 c_1(M)\cdot\Omega>0,\qquad
 \delta_\Omega
 =c_1(M)^2-\frac23\frac{(c_1(M)\cdot\Omega)^2}{V}>0,
 \qquad
 \mathcal C_0<\frac{16\pi^2}{V}\delta_\Omega.
\]
The volume and average K\"ahler scalar curvature are fixed along the flow:
\[
 \operatorname{Vol}(M,g_t)=\frac V2,
 \qquad
 \underline R=\frac{4\pi}{V}c_1(M)\cdot\Omega>0.
\]
This asserts positivity of the average scalar curvature.
By the monotonicity formula \eqref{eq:calabi-monotonicity}, we have
\begin{equation}
 0\le\mathcal C(\varphi(t))\le\mathcal C_0,
 \qquad 0\le t<T.
 \label{eq:Tian-energy-bound}
\end{equation}
Multiplying the strict initial energy bound by $6V>0$ gives
$6V\mathcal C_0<96\pi^2\delta_\Omega$. Hence
\begin{equation}
 \eta_0:=96\pi^2\delta_\Omega-6V\mathcal C_0>0.
 \label{eq:Tian-fixed-energy-gap}
\end{equation}
We record the conversion between the normalised energy and real scalar-curvature integrals. Since $d\mu_t$ is a probability measure and $\int_M(R_t-\underline R)\,d\mu_t=0$,
\[
 \begin{aligned}
 \int_MR_t^2\,d\mu_t
 =\int_M(R_t-\underline R)^2\,d\mu_t
   +2\underline R\int_M(R_t-\underline R)\,d\mu_t
   +\underline R^2\int_Md\mu_t
 =\mathcal C(\varphi(t))+\underline R^2.
 \end{aligned}
\]
Using $R_{\mathrm{real}}(g_t)=2R_t$ and $dV_t=(V/2)d\mu_t$, we obtain
\begin{equation}
 \int_MR_{\mathrm{real}}(g_t)^2\,dV_t
 =4\int_MR_t^2\,dV_t
 =2V\bigl(\mathcal C(\varphi(t))+\underline R^2\bigr).
 \label{eq:Tian-real-scalar-energy}
\end{equation}
Let \[ \sigma_t=
 \left(\int_M
   \bigl(R_{\mathrm{real}}(g_t)-2\underline R\bigr)^2\,dV_t
 \right)^{1/2}.
\]
The same conversion gives
\begin{equation}
 \sigma_t^2
 =4\int_M(R_t-\underline R)^2\,dV_t
 =2V\mathcal C(\varphi(t))\le2V\mathcal C_0.
 \label{eq:Tian-real-fluctuation}
\end{equation}

\textbf{Step 2: A lower bound for the squared Yamabe constant.}
For completeness, the argument from \cite[Section~5]{CLW} is included below.
For a smooth Riemannian metric $q$ on the closed four-manifold $M$, we define its Yamabe constant by
\begin{equation}
 \mathcal Y(q)=
 \inf_{\substack{f\in C^\infty(M,\mathbb R)\\f\not\equiv0}}
 \frac{\displaystyle\int_M
   \bigl(6|\nabla f|_q^2+R_{\mathrm{real}}(q)f^2\bigr)\,dV_q}
 {\displaystyle\left(\int_M|f|^4\,dV_q\right)^{1/2}}.
 \label{eq:Yamabe-real-definition}
\end{equation}
All gradients and Laplacians in this step are real Riemannian ones. The conformal class is
\[
 [q]=\{u^2q:u\in C^\infty(M),\ u>0\}.
\]
With our nonpositive-spectrum Laplacian convention, the conformal transformation laws are
\[
 R_{\mathrm{real}}(u^2q)u^3
 =-6\Delta_q^{\mathrm R}u+R_{\mathrm{real}}(q)u,
 \qquad dV_{u^2q}=u^4dV_q.
\]
The coefficient $6=4(m-1)/(m-2)$ and the power
$4/(m-2)=2$ in $u^2q$ are the specialisations to real dimension
$m=4$. For these transformation laws and the equivalent Yamabe quotient, see \cite[Section~5]{CLW}, where the Laplacian has the opposite sign convention.
For $u>0$, we multiply the scalar-curvature formula by $u$ and integrate to obtain
 \begin{align*}
 \int_M R_{\mathrm{real}}(u^2q)\,dV_{u^2q}
 =\int_Mu\bigl(-6\Delta_q^{\mathrm R}u
                  +R_{\mathrm{real}}(q)u\bigr)\,dV_q
 =\int_M\bigl(6|\nabla u|_q^2
                  +R_{\mathrm{real}}(q)u^2\bigr)\,dV_q.
 \end{align*}
Also $\operatorname{Vol}(M,u^2q)=\int_Mu^4\,dV_q$. Thus the quotient in \eqref{eq:Yamabe-real-definition}, for positive $u$, is the normalised total scalar curvature of $u^2q$.
Allowing arbitrary real smooth $f$ does not change the infimum. Indeed, the positive functions form a subclass, so their infimum cannot be smaller. Conversely, for a nonzero smooth real-valued function $f$, let
$u_\varepsilon=(f^2+\varepsilon^2)^{1/2}>0$. Then
\[
 |\nabla u_\varepsilon|_q^2
 =\frac{f^2}{f^2+\varepsilon^2}|\nabla f|_q^2
 \le |\nabla f|_q^2,
 \qquad u_\varepsilon^2=f^2+\varepsilon^2.
\]
The scalar-curvature part of the numerator converges to that for
$f$, and the denominator converges to
$(\int_M|f|^4\,dV_q)^{1/2}>0$. Therefore the limit superior of the
quotients for $u_\varepsilon$ is at most the quotient for $f$. Taking the infimum over $f$ proves the reverse inequality.
We use the four-dimensional form of the Yamabe existence theorem
\cite{Schoen}, also recalled in \cite[Section~5]{CLW}: for every
smooth metric $q$ on a closed connected four-manifold, the infimum in \eqref{eq:Yamabe-real-definition} is attained by a smooth positive function. By constant rescaling this function can be chosen so that
\[
 \int_Mu^4\,dV_q=1,
 \qquad
 -6\Delta_q^{\mathrm R}u+R_{\mathrm{real}}(q)u
       =\mathcal Y(q)u^3.
\]
Equivalently, $q_Y=u^2q$ has unit volume and constant real scalar curvature equal to $\mathcal Y(q)$. The scalar curvature equals the Yamabe constant, rather than just an unspecified constant, because this is a minimising metric with unit volume.
Let $\mathcal Y_t=\mathcal Y(g_t)$. For each fixed $t<T$, the metric
$g_t$ is smooth on the closed four-manifold $M$, so the theorem gives
a positive smooth $u_t$ and
\[
 q_t=u_t^2g_t,\qquad
 \operatorname{Vol}(M,q_t)=1,\qquad
 R_{\mathrm{real}}(q_t)\equiv\mathcal Y_t.
\]
For a complex surface, $c_1(M)^2=2\chi(M)+3\tau(M)$ with the complex orientation. Adding the two displayed curvature identities in \eqref{CGB} cancels the $|W_q^-|^2$ terms and gives
\begin{equation}
 4\pi^2c_1(M)^2
 =\int_M\left(
  2|W_q^+|_q^2+\frac1{24}R_{\mathrm{real}}(q)^2
       -\frac12|E_q|_q^2\right)dV_q.
 \label{eq:Tian-Chern-Weil}
\end{equation}

We next compare the Weyl integrals of $q_t$ and $g_t$. Under
$q\mapsto u^2q$ in real dimension four, the Hodge star on two-forms
is unchanged: its conformal weight is $u^{4-2\cdot2}=1$. Thus the self-dual subbundle is the same for the two metrics. Conformal covariance of the Weyl tensor gives
\[
 |W_{u^2q}^+|_{u^2q}^2=u^{-4}|W_q^+|_q^2,
 \qquad dV_{u^2q}=u^4dV_q.
\]
The two factors cancel, so
\[
 \int_M|W_{q_t}^+|_{q_t}^2\,dV_{q_t}
 =\int_M|W_{g_t}^+|_{g_t}^2\,dV_t.
\]

We apply this K\"ahler identity \eqref{CGB W+} only to $g_t$, not to the generally non-K\"ahler minimiser $q_t$. Combining it with conformal invariance and the unit-volume constant-scalar-curvature properties of $q_t$ gives
\[
 \begin{aligned}
 \int_MR_{\mathrm{real}}(q_t)^2\,dV_{q_t}
 =\mathcal Y_t^2,\quad
 2\int_M|W_{q_t}^+|_{q_t}^2\,dV_{q_t}
 =\frac1{12}\int_MR_{\mathrm{real}}(g_t)^2\,dV_t.
 \end{aligned}
\]
Substitution in \eqref{eq:Tian-Chern-Weil} yields
\[
 4\pi^2c_1(M)^2
 =\frac{\mathcal Y_t^2}{24}
    +\frac1{12}\int_MR_{\mathrm{real}}(g_t)^2\,dV_t
    -\frac12\int_M|E_{q_t}|_{q_t}^2\,dV_{q_t}.
\]
Multiplying by $24$, solving for $\mathcal Y_t^2$, and discarding the nonnegative Ricci term on the right-hand side, we obtain
\begin{equation}
 \begin{split}
 \mathcal Y_t^2
 &=96\pi^2c_1(M)^2
      -2\int_MR_{\mathrm{real}}(g_t)^2\,dV_t
      +12\int_M|E_{q_t}|_{q_t}^2\,dV_{q_t}\\
 &\ge96\pi^2c_1(M)^2
      -2\int_MR_{\mathrm{real}}(g_t)^2\,dV_t.
 \end{split}
 \label{eq:Tian-Yamabe-square}
\end{equation}

\textbf{Step 3: The squared Yamabe gap from the curvature identities.}
By the fixed average-curvature formula,
\[
 4V\underline R^2
 =4V\left(\frac{4\pi}{V}c_1(M)\cdot\Omega\right)^2
 =64\pi^2\frac{(c_1(M)\cdot\Omega)^2}{V}.
\]
Hence the definition of $\delta_\Omega$ gives the exact equality
\[
 96\pi^2c_1(M)^2-4V\underline R^2=96\pi^2\delta_\Omega.
\]
Using \eqref{eq:Tian-real-scalar-energy} in
\eqref{eq:Tian-Yamabe-square}, we obtain
\begin{equation}
 \mathcal Y_t^2
 \ge96\pi^2c_1(M)^2-4V\underline R^2
                         -4V\mathcal C(\varphi(t))
 =96\pi^2\delta_\Omega-4V\mathcal C(\varphi(t)).
 \label{eq:Tian-Yamabe-cone}
\end{equation}
Subtracting $\sigma_t^2=2V\mathcal C(\varphi(t))$, as given by \eqref{eq:Tian-real-fluctuation}, from \eqref{eq:Tian-Yamabe-cone}, and then using \eqref{eq:Tian-fixed-energy-gap}, we obtain
\begin{equation}
 \mathcal Y_t^2-\sigma_t^2
 \ge96\pi^2\delta_\Omega-6V\mathcal C(\varphi(t))
 \ge\eta_0>0.
 \label{eq:Tian-squared-gap}
\end{equation}

\textbf{Step 4: A lower bound for the Yamabe constant.}
A lower bound for $\mathcal Y_t^2$ alone does not imply
$\mathcal Y_t>0$. It is therefore necessary to rule out the nonpositive case using $\underline R>0$.
For any nonzero smooth real-valued function $f$, the Cauchy--Schwarz inequality gives
\[
 \begin{aligned}
 \left|2\int_M(R_t-\underline R)f^2\,dV_t\right|
 \le
 \left(\int_M[2(R_t-\underline R)]^2\,dV_t\right)^{1/2}
 \left(\int_M(f^2)^2\,dV_t\right)^{1/2}
 =\sigma_t\left(\int_M|f|^4\,dV_t\right)^{1/2}.
 \end{aligned}
\]
In the numerator of the Yamabe quotient, we write $2R_t=2\underline R+2(R_t-\underline R)$. Since the gradient term and $2\underline R\int_Mf^2dV_t$ are nonnegative,
\[
 \begin{aligned}
 \int_M\bigl(6|\nabla f|_{g_t}^2+2R_tf^2\bigr)\,dV_t
 &\ge6\int_M|\nabla f|_{g_t}^2\,dV_t
       +2\underline R\int_Mf^2\,dV_t
       -\sigma_t\left(\int_M|f|^4\,dV_t\right)^{1/2}\\
 &\ge-\sigma_t\left(\int_M|f|^4\,dV_t\right)^{1/2}.
 \end{aligned}
\]
Dividing by the positive denominator and taking the infimum over $f$, we obtain from \eqref{eq:Yamabe-real-definition}
\[
 \mathcal Y_t\ge-\sigma_t.
\]
If $\mathcal Y_t\le0$, then
$-\sigma_t\le\mathcal Y_t\le0$, and therefore
$\mathcal Y_t^2\le\sigma_t^2$, contradicting
\eqref{eq:Tian-squared-gap}. Thus
\[
 \mathcal Y_t>0.
\]
Since $\sigma_t\ge0$ and $\mathcal Y_t^2>\sigma_t^2$, it follows that $\mathcal Y_t>\sigma_t$.
To obtain a bound independent of time, we define
\begin{equation}
 \gamma_0
 =\sqrt{96\pi^2\delta_\Omega-4V\mathcal C_0}
       -\sqrt{2V\mathcal C_0}.
 \label{eq:Tian-uniform-gap-constant}
\end{equation}
The energy gap gives
\[
 96\pi^2\delta_\Omega-4V\mathcal C_0
 =2V\mathcal C_0+\eta_0>2V\mathcal C_0\ge0.
\]
Both square roots are well defined and the first is strictly larger. From \eqref{eq:Tian-fixed-energy-gap}, we further have
\[
 \gamma_0
 =\frac{\eta_0}
 {\sqrt{96\pi^2\delta_\Omega-4V\mathcal C_0}
                    +\sqrt{2V\mathcal C_0}}>0.
\]
We can now take the positive square root in
\eqref{eq:Tian-Yamabe-cone}. In combination with
\eqref{eq:Tian-real-fluctuation}, this yields
\begin{equation}
 \begin{split}
 \mathcal Y_t-\sigma_t
 &\ge\sqrt{96\pi^2\delta_\Omega-4V\mathcal C(\varphi(t))}
       -\sqrt{2V\mathcal C(\varphi(t))}\\
 &\ge\sqrt{96\pi^2\delta_\Omega-4V\mathcal C_0}
       -\sqrt{2V\mathcal C_0}
 =\gamma_0.
 \end{split}
 \label{eq:Tian-uniform-Yamabe-gap}
\end{equation}
The second inequality follows term by term from
$\mathcal C(\varphi(t))\le\mathcal C_0$.

\textbf{Step 5: The uniform Sobolev inequality.}
We follow the argument in the proof of
\cite[Lemma~2.3]{ChenHeToric}.
By the definition of an infimum in \eqref{eq:Yamabe-real-definition}, every nonzero real smooth $f$ satisfies
\[
 \mathcal Y_t\left(\int_M|f|^4\,dV_t\right)^{1/2}
 \le6\int_M|\nabla f|_{g_t}^2\,dV_t
       +2\int_MR_tf^2\,dV_t.
\]
Decomposing $R_t$ into its mean and fluctuation and applying the Cauchy--Schwarz estimate from \textbf{Step~4}, we obtain
\[
 \begin{aligned}
 \mathcal Y_t\left(\int_M|f|^4\,dV_t\right)^{1/2}
 &\le6\int_M|\nabla f|_{g_t}^2\,dV_t
       +2\underline R\int_Mf^2\,dV_t
       +2\int_M(R_t-\underline R)f^2\,dV_t\\
 &\le6\int_M|\nabla f|_{g_t}^2\,dV_t
       +2\underline R\int_Mf^2\,dV_t
       +\sigma_t\left(\int_M|f|^4\,dV_t\right)^{1/2}.
 \end{aligned}
\]
After moving the final term to the left, we obtain the coefficient
$\mathcal Y_t-\sigma_t$, which is at least $\gamma_0>0$ by
\eqref{eq:Tian-uniform-Yamabe-gap}. Division therefore gives
\begin{equation}
 \left(\int_M|f|^4\,dV_t\right)^{1/2}
 \le\frac6{\gamma_0}\int_M|\nabla f|_{g_t}^2\,dV_t
       +\frac{2\underline R}{\gamma_0}\int_Mf^2\,dV_t,
 \qquad 0\le t<T.
 \label{eq:Tian-volume-Sobolev}
\end{equation}
The case $f\equiv0$ is immediate.
To match Definition~\ref{defn:Sobolev}, we convert the three integrals in \eqref{eq:Tian-volume-Sobolev} to $d\mu_t$:
 \begin{align*}
\int_M|f|^4\,dV_t
 =V/2\int_M|f|^4\,d\mu_t,\quad
 \int_M|\nabla f|_{g_t}^2\,dV_t
 =\frac V2\int_M|\nabla f|_{g_t}^2\,d\mu_t,\quad
 \int_Mf^2\,dV_t=\frac V2\int_Mf^2\,d\mu_t.
 \end{align*}
Substituting these equalities and dividing by $\sqrt{V/2}>0$, we obtain
\[
 \left(\int_M|f|^4\,d\mu_t\right)^{1/2}
 \le\frac{\sqrt{V/2}}{\gamma_0}
 \left(6\int_M|\nabla f|_{g_t}^2\,d\mu_t
                 +2\underline R\int_Mf^2\,d\mu_t\right).
\]
Let \begin{equation} C_S^*=
 \frac{\sqrt{V/2}}{\gamma_0}\max\{6,2\underline R\}.
 \label{eq:Tian-Sobolev-constant}
\end{equation}
Both integrals on the right are nonnegative, so
\begin{equation}
 \left(\int_M|f|^4\,d\mu_t\right)^{1/2}
 \le C_S^*\left(
 \int_M|\nabla f|_{g_t}^2\,d\mu_t+
 \int_M|f|^2\,d\mu_t\right),
 \qquad 0\le t<T.
 \label{eq:Tian-probability-Sobolev}
\end{equation}
For completeness, we fix $t<T$ and let $f\in W^{1,2}(M,g_t)$. Let $f_j$ be smooth real-valued functions converging to $f$ in $W^{1,2}(M,g_t)$. Applying \eqref{eq:Tian-probability-Sobolev} to $f_j-f_k$, we obtain
\[
 \|f_j-f_k\|_{L^4(\mu_t)}^2
 \le C_S^*\left(
  \|\nabla(f_j-f_k)\|_{L^2(\mu_t)}^2
          +\|f_j-f_k\|_{L^2(\mu_t)}^2\right)\longrightarrow0.
\]
Hence $f_j$ has an $L^4$ limit. Since $\mu_t(M)=1$, $L^4$ convergence implies $L^2$ convergence, and this limit must equal $f$. Passing to the limit in the inequality for $f_j$ proves
\eqref{eq:Tian-probability-Sobolev} for $f$.
Therefore the least constant in \eqref{eq:Sobolev} satisfies
\[
 C_S(\varphi(t))\le C_S^*,\qquad 0\le t<T.
\]
The number $C_S^*$ depends only on $V$, $c_1(M)\cdot\Omega$,
$\delta_\Omega$, and $\mathcal C_0$. It is independent of $t$ and
of the maximal existence time $T$.
Finally, suppose that $T<\infty$. By the estimate just proved, we can apply Theorem~\ref{thm:main} on the entire existence interval and extend the original Calabi flow smoothly beyond $T$, with initial potential $\varphi_0$. This contradicts the maximality of $T$, and hence $T=\infty$, completing the proof of Corollary~\ref{cor:Tian-main}.
\end{proof}

\vskip10pt
Haozhao Li, Institute of Geometry and Physics, and  School of Mathematical Sciences, University of Science and Technology of China, No. 96 Jinzhai Road, Hefei, Anhui Province, 230026, China.  Email: hzli@ustc.edu.cn\par\medskip

Bing Wang, Institute of Geometry and Physics, and School of Mathematical Sciences, University of Science and Technology of China, Hefei 230026, China, and Hefei National Laboratory, Hefei 230088, China. Email: topspin@ustc.edu.cn\par\medskip

Kai Zheng, University of Chinese Academy of Sciences, Beijing, 100190, China. Email: KaiZheng@amss.ac.cn

\end{document}